\documentclass[12pt]{article}

\usepackage{amssymb,amsmath,amsfonts,amssymb}   
\usepackage{graphics,graphicx,color,hhline}
         
\def\@abssec#1{\vspace{.05in}\footnotesize \parindent .2in 
{\bf #1. }\ignorespaces}

\graphicspath{{/EPSF/}{Figures/}}
   
\newtheorem{theorem}{Theorem}[section]
\newtheorem{lemma}[theorem]{Lemma}
\newtheorem{proposition}[theorem]{Proposition}

\newtheorem{remark}[theorem]{Remark}

\def \Rm {\mathbb R}

\newcommand{\eps}{\varepsilon}
 
\newcommand{\NN}{\mathbb N }

\newcommand{\dint}{\displaystyle\int}  

\newcommand{\bn}{\mathbf n}

\newcommand{\bx}{\mathbf x} \newcommand{\by}{\mathbf y}
\newcommand{\bz}{\mathbf z} 

 \newcommand{\bF}{\mathbf F}

 \newcommand{\bS}{\mathbf S}

\newcommand{\calC}{\mathcal C}

\newcommand{\calI}{\mathcal I}
\newcommand{\calL}{\mathcal L}

\newcommand{\calO}{\mathcal O}

\newcommand{\calM}{\mathcal M}

\newcommand{\calD}{\mathcal D}

\newcommand{\cout}[1]{}

\newcommand{\DST}{\displaystyle}

\def\un{{\mathbf{1}}\hspace{-0.24em}\mathrm{I}}
 \renewcommand{\arraystretch}{1.5}

\title{Small volume expansions for elliptic equations}

\author{Guillaume Bal \thanks{Department of Applied Physics and
    Applied Mathematics, Columbia University, New York NY, 10027;
    gb2030@columbia.edu; 
   } \and Olivier Pinaud
  \thanks{Universit\'e de Lyon,
Universit\'e Lyon 1, CNRS, UMR 5208 Institut Camille Jordan/ISTIL,
B\^atiment du Doyen Jean Braconnier,
43, blvd du 11 novembre 1918,
F - 69622 Villeurbanne Cedex, 
France; 
  pinaud@math.univ-lyon1.fr} }

\begin{document}
 
\maketitle


\begin{abstract}
  This paper analyzes the influence of general, small volume,
  inclusions on the trace at the domain's boundary of the solution to
  elliptic equations of the form $\nabla \cdot D^\eps \nabla u^\eps=0$
  or $(-\Delta + q^\eps) u^\eps=0$ with prescribed Neumann conditions.
  The theory is well-known when the constitutive parameters in the
  elliptic equation assume the values of different and smooth
  functions in the background and inside the inclusions. We generalize
  the results to the case of arbitrary, and thus possibly rapid,
  fluctuations of the parameters inside the inclusion and obtain
  expansions of the trace of the solution at the domain's boundary up
  to an order $\eps^{2d}$, where $d$ is dimension and $\eps$ is the
  diameter of the inclusion.  We construct inclusions whose leading
  influence is of order at most $\eps^{d+1}$ rather than the expected
  $\eps^d$. We also compare the expansions for the diffusion and
  Helmholtz equation and their relationship via the classical
  Liouville change of variables.
\end{abstract}


\renewcommand{\thefootnote}{\fnsymbol{footnote}}
\renewcommand{\thefootnote}{\arabic{footnote}}

\renewcommand{\arraystretch}{1.1}
\paragraph{keywords:}
  Small volume inclusions, asymptotic expansions, diffusion equation,
  Helmholtz equation, inverse problems.

\paragraph{AMS:} 35J15, 35B40, 65R20, 35R30.


\section{Introduction} 
\label{sec:intro} 

Asymptotic expansions for the influence of small volume inclusions for
elliptic and other equations is now well-established. We refer the
reader to e.g.
\cite{AK-04,SmallAbso-02,CapVogM2AN1-02,CMV-IP,FriedVog89,HPV-preprint}
and their references for a few historic and recent works on the
subject. A major advantage of such expansions is that they help us
understand what details of the constitutive parameters in the equation
may or may not be reconstructed from available boundary measurements.
Indeed, in the elliptic equations of interest in this paper, namely
the diffusion or conductivity equation and the Helmholtz equation, the
reconstruction of the constitutive parameters $X$ from knowledge of
the full Dirichlet-to-Neumann map $\Lambda$, the most general type of
information available at the domain's boundary, is an extremely
ill-conditioned problem. Available stability estimates for both types
of equations predict that the accuracy in the reconstruction is at
best logarithmic in the accuracy of the measurements. More precisely,
we have \cite{Al-JDE90,isakov-98}
$$\|X_1-X_2\|_{L^\infty(\Omega)} \leq C \Big| \log \| \Lambda_1 - \Lambda_2 \|_{\calL(H^{\frac 12}(\partial \Omega),H^{-\frac{1}2}(\partial \Omega) )}\Big|^{-\delta},
$$      
for some positive constant $C$ and $\delta \in (0,1)$, where $X_1$
and $X_2$ are two sets of parameters and $\Lambda_1$ and $\Lambda_2$
their corresponding measurements.     
 
For such severely ill-posed problems, only a limited number of degrees
of freedom may be reconstructed from even quite accurate measurements.
A natural way of limiting the number of degrees of freedom is to
assume that the constitutive coefficients are known throughout the
domain, except at some locations where unknown inclusions may be
present. The asymptotic expansions in the size of the inclusion
mentioned above thus provide a very efficient tool to understand what
may or may not be reconstructed from data with a given level of noise.
  
For elliptic equations, the existing works on the subject, see e.g.
\cite{AK-04,CapVogM2AN1-02}, typically assume the parameters jumps
across the interface of the inclusion. One of the main objectives of
this paper is to consider the case of more general inclusions whose
coefficient may vary at the small scale $\eps$ and need not ``jump''
from the values of the background parameters. We also want to stress
the similarities and differences between expansions for the diffusion
equation $\nabla\cdot D^\eps\nabla u^\eps=0$ and the Helmholtz
equation $(-\Delta+q^\eps)u^\eps=0$ with $q^\eps$ of order
$\eps^{-2+\eta}$ for $\eta\in[0,2]$. In both cases of the diffusion
equation and the Helmholtz equation when $\eta=0$, we need to
introduce local correctors and obtain a limiting influence at the
domain's boundary that is non-linear in the parameters inside the
inclusion.

The maximal leading term in the expansion is always of order
$\calO(\eps^d)$, the volume of the inclusion. We construct expansions
up to the order $\eps^{2d}$. Going beyond this order of accuracy
requires a more careful analysis of the decay properties of local
correctors at infinity than is available here, or the use of single
and double layer potentials as in \cite{AK-04} in the case of constant
coefficients inside and outside of the inclusion. Note that the
cross-talk between two inclusions of volume $\calO(\eps^d)$ is also a
term of order $\eps^{2d}$.  It seems therefore natural to stop the
expansion at the order $\calO(\eps^{2d})$ for the influence of any
given well-separated inclusions.

Because our inclusions are modeled by somewhat arbitrary parameters
that need not jump from the local value of the background parameter or
are not constant, the limiting polarization tensors need not satisfy
any property of positivity or definiteness. On the contrary, we show
that the polarization tensors vanish to first order for some types of
inclusions, whose influence at the domain's boundary is therefore at
most of order $\eps^{d+1}$ rather than $\eps^d$. Although we do not
explore this aspect here, the proposed asymptotic expansions may be
used to construct inclusions whose influence on the measurements is
minimized in a prescribed manner.

The rest of the paper is structured as follows. Section \ref{sec:diff}
is devoted to the derivation of the asymptotic expansions for the
diffusion equation. The main tool in the expansion is a decomposition
of the corresponding Green's function given in proposition
\ref{prop:decompNdiff}. The expansion obtained for smooth inclusions
is presented in theorem \ref{th:asymp} while the generalization to
more singular inclusions with possible discontinuities of the
coefficients across the inclusion's boundary is given in theorem
\ref{th2}. We compare our expansions with those obtained in
\cite{AK-04} for constant coefficients inside and outside of the
inclusions in proposition \ref{prop:jump}. Section \ref{subsec:prop}
presents some properties of the polarization tensors that appear in
the asymptotic expansions. In particular, proposition
\ref{prop:vanish} shows that the leading polarization tensor vanishes
for some non-vanishing diffusion coefficients inside the inclusion.
Some proofs of the results are postponed until section \ref{sec:proofth}.

Section \ref{sec:helm} addresses local variations of the potential in
a Helmholtz equation. The appropriate decomposition of the Green's
function is shown in proposition \ref{prop:decompN} and the main
result in theorem \ref{th:asympV}. The relationship between the
expansions for diffusion and Helmholtz equations in regards of the
Liouville change of variables is explored in section \ref{sec:rel}. We
show that the expansions in both settings agree up to order
$\eps^{d+2}$. Most proofs are postponed to section \ref{sec:proofth}.

\section{Perturbations of the diffusion problem}
\label{sec:diff}
In this section, we are interested in the analysis of small inclusions
in the diffusion or conductivity problem.  As we have mentioned in the
introduction, the reconstruction of diffusion or conductivity
coefficients from boundary measurements is a severely ill-posed
problem. One possible way to overcome this difficulty is to assume
that the background diffusion coefficient is known and that the
unknown part of the coefficient is localized and has small volume.

Under such hypotheses, asymptotic expansions of the perturbed field in
the volume of the inclusion have been derived in \cite{FriedVog89}
when the inclusion is perfectly reflecting or insulating. These
formulas have then been extended to more general inclusions in
\cite{CMV-IP}, and to higher orders in the volume and to domain with
Lipschitz boundaries in \cite{AK-04}. In those references, the
inclusion is modeled by a jump in the diffusion coefficient so that
its first order effect on the boundary measurements is proportional to
the inclusion's volume. The so-called polarization tensor contains the
information about the inclusion that is available at this level of the
asymptotic expansion. 

Such a setting for the diffusion coefficient prevents us from using
the well-known change of variable $q:=\frac{\Delta
  \sqrt{D}}{\sqrt{D}}$ that allows us to relate the diffusion equation
to the Helmholtz or Schr\"odinger equation.  Since one of the
objective of the paper is to show the equivalence of the asymptotic
expansions within the diffusion and Helmholtz frameworks, we first
consider a regular inclusion without jump and derive the corresponding
asymptotic expansions in section \ref{sec:smooth}.  We next generalize
these formulas to the case with jumps in the diffusion coefficient in
section \ref{subsec:non-smooth}. We also recover the formulas in
\cite{AK-04} in the special case of constant coefficients in the
background and the inclusion. Finally, we present in section
\ref{subsec:prop} some properties the polarization tensors involved in
the asymptotic formula.

\subsection{The case of smooth inclusions}
\label{sec:smooth}
We consider the following system of equations:
\begin{equation}
\label{eq:diff}
 \left\{ 
\begin{array}{rll}
 \nabla \cdot D^\eps \nabla u^\eps=0, \quad &\textrm{in}\; \Omega,\\
\DST D^\eps \frac{\partial u^\eps}{\partial \bn }=g, \quad 
 &\textrm{on}\; \partial \Omega, \quad&
\DST \int_{\partial \Omega} u^\eps d \sigma=0,
\end{array} \right.
\end{equation}
where $\Omega$ is a bounded open domain of dimension $d \geq 2$ with
Lipschitz boundary, $\sigma$ is the surface measure on $\partial
\Omega$, and $ g\in L^2(\partial \Omega)$ such that the following
compatibility condition holds $\int_{\partial\Omega}g d\sigma=0$. It
is assumed that $D^\eps$ is bounded from below by a positive constant
independent of $\eps$ and that $D^\eps$ satisfies the decomposition
$D^\eps (\bx) = D_0(\bx)+D_1(\frac{\bx-\bx_0}{\eps})$, where $0<
C_0'\leq D_0 \in \calC^\infty(\overline{\Omega})$, $D_1 \in
L^{\infty}(\Omega)$ and $D_1$ vanishing in $\Rm^d \backslash
\overline{B}$, $B$ being a bounded set with Lipschitz boundary. The
properties of $D^\eps$ are summarized below:
 \begin{equation}
\label{eq:hypD}
 \left\{ 
\begin{array}{ll}
 D^{\eps}(\bx)\geq C_0>0, \quad &\Omega \textit{ a.e.}, \\
D^\eps (\bx) = D_0(\bx), \qquad &\bx \in \overline{\Omega} \backslash \overline{\bx_0+\eps B},\\
D^\eps (\bx) = D_0(\bx)+D_1(\frac{\bx-\bx_0}{\eps}),\qquad& \bx \in \bx_0+\eps B,\\
D_0 \in \calC^\infty(\overline{\Omega}), \qquad D_1 \in L^\infty(\Omega).
\end{array}
\right. 
\end{equation}
We assume in addition that the domain of the inclusion is located away
from the boundary in the sense that there exists $d_0>0$
independent of $\eps$ such that
\begin{equation} \label{hypB}\textrm{dist}(\partial \Omega, \bx_0+\eps B)>d_0. \end{equation}
The Lax-Milgram lemma applied to (\ref{eq:diff})-(\ref{eq:hypD})
yields a unique variational solution $u^\eps \in H^{1}(\Omega)$. Let
us denote by $U$ the solution with background diffusion coefficient
$D_0$:
\begin{equation}
\label{eq:diffD0}
 \left\{ 
\begin{array}{rll}
 \nabla \cdot D_0 \nabla U=0, \quad &\textrm{in}\; \Omega,\\
\DST D_0\frac{\partial U}{\partial \bn }=g, \quad &\textrm{on}\; \partial \Omega,\qquad&
\DST \int_{\partial \Omega} U(\bx) d \sigma(\bx)=0,
\end{array} \right. 
\end{equation}
and introduce the related Green function $N \in \calD'(\Omega\times
\Omega)$ satisfying, for all fixed $\by$ in $\Omega$,
\begin{equation}
\label{eq:N}
 \left\{ 
\begin{array}{rll}
 \nabla_\bx \cdot D_0(\bx) \nabla_\bx N(\bx,\by)=-\delta(\bx-\by), 
 \quad &\textrm{in}\; \Omega,\\[2mm]
\DST D_0(\bx)\frac{\partial N(\bx,\by)}{\partial \bn_\bx }
 =-\frac{1}{|\partial \Omega|}, \quad &\textrm{on}\; \partial \Omega,\quad&
\DST \int_{\partial \Omega} N(\bx,\by) d \sigma(\bx)=0.
\end{array} \right. 
\end{equation}
For all $\bx \in \overline{\Omega}$, the Lax-Milgram lemma yields
again a unique variational solution $U \in H^{1}(\Omega)$ and standard
elliptic regularity results \cite{gt1} implies that $U \in
\calC^\infty(\Omega)$ since $D_0 \in \calC^\infty(\overline{\Omega})$.
We denote by $\Gamma$ the fundamental solution of the Laplacian,
namely
\begin{equation}
\label{eq:G}
\Gamma(\bx)= \left\{ 
\begin{array}{ll}
\DST -\frac{1}{2 \pi} \log |\bx|, \qquad d=2,\\
\DST \frac{1}{(d-2) |S_{d-1}|} \frac{1}{|\bx|^{d-2}}, \qquad d\geq 3,
\end{array} \right.
\end{equation}
where $|S_{d-1}|$ is the measure of the $(d-1)$-dimensional unit
sphere. Throughout the paper, we use the following multi-index
notations: for $i=(i_1,\cdots,i_d) \in \NN^d$, we define
$|i|=i_1+\cdots+i_d$, $\partial ^i f=\partial^{i_1}_1 f \cdots
\partial^{i_1}_d f$ and $\bx^i=x_1^{i_1}\cdots x_d^{i_d}$.  We also
define $i!=i_1!\cdots i_d!$.

One of the main tools in our asymptotic expansions is the following
decomposition of the Green function $N$:
\begin{proposition} \label{prop:decompNdiff}
  The Green function $N$ can be decomposed, for $(\bx,\by) \in \Omega
  \times \Omega$, as:
\begin{equation} \label{decompNdiff1}
N(\bx,\by)=D_0^{-1}(\bx) \Gamma(\bx-\by)+R_1(\bx,\by)+R_2(\bx,\by)+R_3(\by),
\end{equation}
where $R_3 \in \calC^\infty(\Omega)$; for all $\by$ fixed in $\Omega$,
$R_1(\cdot,\by) \in W^{1,p}(\Omega)$, with $1 \leq p < \frac{d}{d-2}$
when $d\geq 3$ and $p<\infty$ when $d=2$; and $R_2(\cdot,\by) \in
H^1(\Omega)$.  Moreover, $R_1$ is $\calC^\infty$ when $\bx \neq \by$,
$R_2 \in \calC^\infty(\Omega \times \Omega)$, and we have by
construction that:
\begin{equation} \label{decompNdiff2}
\nabla_\bx N(\bx,\by) =D_0^{-1}(\bx) \nabla_\bx \Gamma(\bx-\by)+\nabla_\bx R_2(\bx,\by).
\end{equation} 
Also, $N$ admits the following asymptotic expansion for $\bx \in B$,
$\by$ \textit{a.e.}  in $\partial \Omega$:
\begin{equation} \label{decompNdiff3}
\nabla _\bx N(\bx_0+\eps \bx,\by)=\sum_{|i|=1}^{d} \frac{\eps^{|i|}}{i!}  \nabla \bx^i \partial^i_\bx N(\bx_0,\by)+\calO(\eps^{d+1}),
\end{equation} 
where $\calO(\eps^{d+1})$ denotes a term bounded in $L^2(\partial
\Omega)$ by $C\eps^{d+1}$ uniformly in $\bx$.
\end{proposition}
\begin{proof}  
Let $R_1$ be (uniquely) defined by
$$\nabla_\bx R_1(\bx,\by)=\frac{\nabla D_0(\bx)}{D_0^2(\bx)}
\Gamma(\bx-\by), \qquad \int_{\partial \Omega} R_1(\bx,\by)
d\sigma(\bx)=0,$$
and $R_3$ be defined as
$$
 |\partial \Omega| R_3(\by)=-\int_{\partial \Omega} D_0^{-1}(\bx) \Gamma(\bx-\by) d \sigma(\bx).
 $$
 Since $D_0>0$, $D_0 \in \calC^\infty(\overline{\Omega})$ and
 $\Gamma \in L^p_\textrm{loc}(\Rm^d)$ for the values of $p$ in the
 proposition, it follows that $R_1(\cdot,\by) \in W^{1,p}(\Omega)$.
 Moreover, $R_1$ is $\calC^\infty$ as soon as $\bx \neq \by$. In the
 same way, $R_3 \in \calC^\infty(\Omega)$ since $\Gamma(\bx) \in
 \calC^\infty(\Rm^d\backslash\{0\})$. We then verify that
 (\ref{decompNdiff1}) leads to (\ref{decompNdiff2}) and that plugging
 (\ref{decompNdiff1}) into (\ref{eq:N}) leads to the system, for
 $\by \in \Omega$:
 $$
\left\{ 
\begin{array}{rll}
 \nabla_\bx  \cdot D_0(\bx) \nabla_\bx R_2(\bx,\by)=0, \quad &\textrm{in}\; \Omega,\\[3mm]
\DST D_0\frac{\partial  R_2(\bx,\by)}{\partial \bn_\bx }=-\frac{1}{|\partial \Omega|}-\frac{\partial  \Gamma(\bx-\by)}{\partial \bn_\bx }, \quad &\textrm{on}\; \partial \Omega,\quad&
\DST \int_{\partial \Omega}  R_2(\bx,\by)d \sigma(\bx)=0,
\end{array} \right.
$$
which admits a unique weak solution thanks to the Lax-Milgram lemma
since we verify that
$\int_{\partial\Omega}(\frac{1}{|\partial\Omega|}+\frac{\partial\Gamma}
{\partial \bn_\bx})d\sigma(\bx)=0$.  Since $\partial_\by^\beta
\frac{\partial \Gamma}{\partial \bn_\bx }(\cdot-\by) \in L^2(\partial
\Omega)$ for any multi-index $\beta$ and $\by \in \Omega$, we deduce
that $\partial_\by^\beta R_2(\cdot,\by) \in H^1(\Omega)$, so that
elliptic regularity yields $\partial_\by^\beta R_2(\cdot,\by)
\in\calC^\infty(\Omega)$, and finally $R_2 \in \calC^\infty(\Omega
\times \Omega)$.

Moreover, $\partial_\by^\beta R_2(\cdot,\by)$ is bounded in
$H^1(\Omega)$ uniformly in $\by$ when $\by \in \Omega' \subset \subset
\Omega$. To prove (\ref{decompNdiff3}), we first remark from
(\ref{decompNdiff1}) that the trace $\left. \partial_\by^\beta
  N(\bz,\by) \right|_{\partial \Omega}$ is defined in $L^2(\partial
\Omega)$ uniformly in $\by$ when $\by \in \Omega'$ since $R_1\in
\calC^\infty(\overline{\Omega}\backslash \overline{\Omega'} \times
\Omega')$, $\partial_\by^\beta R_2(\cdot,\by) \in H^1(\Omega)$
uniformly in $\by \in \Omega'$, and $R_3 \in \calC^\infty(\Omega')$.
This allows us to apply Green's theorem and obtain, for any $(\bz,\by)
\in \Omega' \times \Omega$, that:
$$
R_2(\bz,\by)=- \int_{\partial \Omega} \left(\frac{1}{|\partial
    \Omega|}+\frac{\partial \Gamma(\bx,\by)}{\partial \bn_\bx }\right)
N(\bx,\bz)d\sigma(\bx).
$$
As $\by$ goes to $\partial \Omega$, the boundary integral converges
for Lipschitz domains $\Omega$, see \cite{AK-04}, to
$$
- \frac{1}{|\partial \Omega|} \int_{\partial \Omega}
N(\bx,\bz)d\sigma(\bx)- \textrm{p.v} \int_{\partial \Omega}
\frac{\partial \Gamma(\bx,\by)}{\partial \bn_\bx }N(\bx,\bz)
d\sigma(\bx) +\frac{1}{2} N(\by,\bz), \quad (\bz,\by) \in \Omega'
\times \partial \Omega,
$$
where p.v. stands for the Cauchy principal value and the above
integral operator is bounded in $L^2(\partial \Omega)$. The first term
belongs to $\calC^\infty(\Omega')$ and the second and the third
terms to $\calC^\infty(\Omega')$ with values in $L^2(\partial
\Omega)$. Using \eqref{decompNdiff2}, this allows us to expand $\nabla
_\bx N(\bx_0+\eps \bx,\by)$ and obtain (\ref{decompNdiff3}).
\end{proof}

We first consider the case of a smooth inclusion by adding the
hypothesis that $D_1$ is regular and compactly supported in $B$, that
is $D_1 \in W^{1,\infty}(\Omega)$, with support ${\rm supp} D_1\subset
B$. In such a context, the trace of $D_1$ vanishes on $\partial B$.
We have the following result:
\begin{theorem} \label{th:asymp}
  Assume that $D_1 \in W^{1,\infty}(\Omega)$ with support ${\rm supp}
  D_1\subset B$. Then the solution $u^\eps$ to
  (\ref{eq:diff})-(\ref{eq:hypD}) verifies the following asymptotic
  expansion, \textit{a.e.} on $\partial \Omega$:
\begin{equation*}
  \begin{array}{l}
\DST \left. u^\eps(\by) \right|_{\partial \Omega}=\left.U(\by) \right|_{\partial \Omega}-\sum_{|i|=1}^{d} \sum_{|j|=1}^{d}\frac{\eps^{d-2+|i|+|j|}}{i!j!}  M_{ij} \,\partial^j U (\bx_0)\,  \left. \partial^i_\bx N(\bx_0,\by) \right|_{\partial \Omega} +\calO(\eps^{2d})\\
\DST-\sum_{|i|=1}^{d} \sum_{|j|=1}^{d} \sum_{|k|=0}^{d}\sum_{l=0,\, l+|k|>0 }^d \frac{\eps^{d-2+|i|+|j|+|k|+l}}{i!j!k!l!} M^2_{ijkl}\partial^j U(\bx_0) \left(\partial^k D_0^{-1}\right) (\bx_0) \left. \partial^i_\bx N(\bx_0,\by)\right|_{\partial \Omega},
 \end{array}
\end{equation*}
where $M$ and $M^2$ are generalized polarization tensors given by
\begin{equation}
  \label{eq:tensors}
  \begin{array}{rcl}
M_{ij}&=& \DST \int_{B} D_1\left(\bx \right)  \nabla  (\bx^j+\phi_{j\, 0}^0(\bx)) \cdot \nabla \bx^i d\bx,\qquad i,j \in \NN^d,\\[3mm]
M^2_{ijkl}&=& \DST\int_{B} D_1\left(\bx \right)  \nabla \phi_{jk}^l(\bx) \cdot \nabla \bx^i d\bx, \qquad i,j,k \in \NN^d, \quad l \in \NN,
  \end{array}
\end{equation}
and the functions $\phi_{jk}^l$ are the unique solutions in
$H^1_\textrm{loc}(\Rm^d) \cap \calC^\infty(\Rm^d \backslash
\overline{B})$ to: 
\begin{equation}
  \label{eq:phijkl1}
  \left\{
\begin{array}{l}
\DST \nabla \cdot \left(D_0(\bx_0)+D_1(\bx) \right)\nabla \phi_{jk}^l\,\,=\,\,-\delta_l^0\, \nabla \cdot \left(D_1\left(\bx \right) \bx^k \nabla \bx^j \right)\\
\hspace{2cm}\DST -D_0(\bx_0)\sum_{|m|=1}^{l} \frac{l!\partial^m D_0^{-1} (\bx_0)}{m! (l-|m|)!}  \nabla \cdot \left(D_1\left(\bx \right) \bx^m \nabla \phi_{jk}^{l-|m|}(\bx)\right),\\
\phi_{jk}^l(\bx) = \calO(|\bx|^{1-d})\quad \textrm{as} \quad |\bx|\to \infty. 
\end{array} \right.
\end{equation}
Here, $\delta_l^0$ is the Kronecker symbol and the notation
$\calO(\eps^{2d})$ in the expansion represents a term bounded in
$L^2(\partial \Omega)$ by a constant depending on $\|D_1\|_{L^\infty}$
and on $\|g\|_{L^2(\partial \Omega)}$.
\end{theorem} 
\begin{remark}\rm
  The function $\phi_{jk}^0$ solves the following equation in $\Rm^d$:
\begin{eqnarray*}
\nabla \cdot \left(D_0(\bx_0)+D_1(\bx) \right)\nabla \phi_{jk}^0&=&-\nabla \cdot \left(D_1\left(\bx \right) \bx^k \nabla \bx^j \right),\\
\phi_{jk}^0(\bx) &=& \calO(|\bx|^{1-d})\quad \textrm{as} \quad |\bx|\to \infty, 
\end{eqnarray*}
so that $\phi_{jk}^l$ is computed from $\phi_{jk}^{m}$, $0\leq m<l$,
iteratively.
\end{remark}
\begin{remark} \rm \label{rem:alt}
We may recast the expansion in theorem \ref{th:asymp} as
\begin{equation*}
  \begin{array}{l}
\DST \left. u^\eps(\by) \right|_{\partial \Omega}=\left.U(\by) \right|_{\partial \Omega}-\sum_{|i|=1}^{d} \sum_{|j|=1}^{d}\frac{\eps^{d-2+|i|+|j|}}{i!j!}  M^\eps_{ij} \,\partial^j U (\bx_0)\,  \left. \partial^i_\bx N(\bx_0,\by) \right|_{\partial \Omega} +\calO(\eps^{2d}),
 \end{array}
\end{equation*}
where the $\eps$-dependent tensor $M^\eps$ is given by:
\begin{equation*}
   \begin{array}{rcl}
M^\eps_{ij}&=& \DST \int_{B} D_1\left(\bx \right)  \nabla  (\bx^j+\Psi^\eps_{j}(\bx)) \cdot \nabla \bx^i d\bx,\qquad i,j \in \NN^d,
  \end{array}
\end{equation*}
and the functions $\Psi^\eps_{j}$ are the unique solutions in
$H^1_\textrm{loc}(\Rm^d) \cap \calC^\infty(\Rm^d \backslash
\overline{B})$ to
\begin{displaymath}
  \begin{array}{rll}
   \nabla\cdot\big( 1+D_1(\bx) D_{0}^{-1}(\bx_0+\eps\bx)\big)
  \nabla\Psi^\eps_j &=& - \nabla \cdot\big( D_1(\bx) D_{0}^{-1}(\bx_0+\eps\bx)
  \nabla \bx^j\big),\\
\Psi^\eps_{j}(\bx) &=& \calO(|\bx|^{1-d})\quad \textrm{as} \quad |\bx|\to \infty. 
  \end{array}
\end{displaymath}
The asymptotic expansion of the theorem is then recovered by expanding
$\Psi^\eps_j$ and $D_0^{-1}(\bx_0+\eps\bx)$ in powers of $\eps$. 

There is another equivalent expansion to that of theorem
\ref{th:asymp} up to the order $\eps^{2d}$. We sketch its derivation
in the case where $D_0$ is constant. The right hand side of the
equation for $\Psi^\eps_j$ is equal to $-D_0^{-1} \nabla D_1 \cdot
\nabla \bx^j-D_0^{-1} D_1 \Delta \bx^j$. It turns out that an
appropriate linear combination of $\Delta \bx^j$ is of order
$\eps^{d+1}$, so that we can replace $\Psi^\eps_j$ in the definition of
$M^\eps_{ij}$ by $\Phi_j$ solution to
\begin{displaymath}
  \begin{array}{rll}
   \nabla\cdot\big( 1+D_1(\bx) D_{0}^{-1}\big)
  \nabla\Phi_j &=& - D_{0}^{-1} \nabla D_1(\bx)\cdot \nabla \bx^j,\\
\Phi_{j}(\bx) &=& \calO(|\bx|^{1-d})\quad \textrm{as} \quad |\bx|\to \infty. 
  \end{array}
\end{displaymath}
The appropriate linear combination is deduced from $\Delta
U(\bx_0+\eps \bx)=0$ and from Taylor expanding $U$ so as to obtain:
$$
0=\Delta U(\bx_0+\eps \bx)= \Delta \sum_{|j|=0}^{d}\frac{\eps^{|j|}}{j!}\,\partial^j U (\bx_0) \bx^j+\calO(\eps^{d+1}).
$$
\end{remark}
\begin{remark} \label{rem2}
\rm The leading order in the expansion is given by 
$$
\eps^d \sum_{|i|=|j|=1}M_{ij} \partial^i N(\bx_0, \by) \partial^j U(\bx_0).
$$
The polarization tensor $M^2$ contributes only to higher orders.
The polarization tensor $M$ captures the correction when the
background diffusion coefficient $D_0$ is constant in $\bx_0+\eps B$,
whereas $M^2$ is the correction that needs to be added when $D_0$ is
not constant in $\bx_0+\eps B$.  When $D_0$ is constant in $\bx_0+\eps
B$, then $M^2_{ijkl}=M^2_{ijkl} \delta_l^0$ so that the expansion then
reduces to the classical formula:
\begin{eqnarray*}
u^\eps(\by)&=&U(\by)+\sum_{|i|=1}^{d} \sum_{|j|=1}^{d}\frac{\eps^{d-2+|i|+|j|}}{i!j!}  M_{ij}\partial^i_\bx N(\bx_0,\by) \partial^j U (\bx_0) +\calO(\eps^{2d}).
\end{eqnarray*}
In this case, using the notation of remark \ref{rem:alt},
$\Psi^\eps_j$ no longer depends on $\eps$ and may be identified with
$\phi_{j0}^0$. Note that the latter formula also holds when $D_0$ is
non-constant away from the support of the inclusion $\bx_0+\eps B$ as
remark \ref{rem:alt} makes clear since only the values of $D_0^{-1}$
on the support of $D_1$ are involved.
\end{remark}
The proof of the theorem is given in section \ref{sec:proofth}. Its
main ingredients are the integral formulation of \eqref{eq:diff} and
the decomposition of the Green function given in proposition
\ref{prop:decompNdiff}.  Additional boundary effects, which are not
considered here, appear at the order $\calO(\eps^{2d})$ when the
geometry-dependent corrector $R_2(\bx,\bx_0+\eps \by)$ of proposition
\ref{prop:decompNdiff} is expanded in powers of $\eps$. When $D_0$ is
constant, a proper factorization based on the technique of double
layer potentials allow us to obtain arbitrarily accurate expansions;
see \cite{AK-04}.
\subsection{The case of singular inclusions} 
\label{subsec:non-smooth}
In the preceding section, we assumed that the perturbed diffusion
coefficient was regular.  We may generalize the above theorem to
include the case where $D_1$ is in $L^\infty(\Rm^d)$ with support in
$B$ and with a possibly non-vanishing trace (if it is defined) at the interior
boundary $\partial B$. This generalization is achieved by regularizing
the singular perturbation so that we can use the preceding result and
then by computing the limiting polarization tensors. We have the
following result:
\begin{theorem} \label{th2}
  Assume $D_1$ verifies \eqref{eq:hypD} with no further assumption on
  its interior trace on $\partial B$. Then $u^\eps$ admits the same expansion
  as in theorem \ref{th:asymp} with polarization tensors
  still given by \eqref{eq:tensors}, where
  $\phi_{jk}^{l}$ is now the unique solution in
  $H^1_\textrm{loc}(\Rm^d) \cap \calC^\infty(\Rm^d \backslash
  \overline{B})$ to:
  \begin{equation}
    \label{eq:phijkl2}
    \left\{
\begin{array}{l}
\DST  \Delta \phi_{jk}^{l} =0, \qquad  \bx \in \Rm^d/\overline{B}
\qquad \mbox{ with }\qquad 
\DST \phi_{jk}^l(\bx) = \calO(|\bx|^{1-d})\quad \textrm{as} \quad |\bx|\to \infty,
 \\[5mm]
\DST \nabla \cdot (D_0(\bx_0)+D_1(\bx)) \nabla \phi_{jk}^{l} =-\delta_l^0\, \nabla \cdot \left(D_1\left(\bx \right) \bx^k \nabla \bx^j \right)\\[0mm]
\quad \DST -D_0(\bx_0)\sum_{|m|=1}^{l}  \frac{l!\partial^m D_0^{-1} (\bx_0)}{m! (l-|m|)!} \nabla \cdot \left(D_1\left(\bx \right) \bx^m \nabla \phi_{jk}^{l-|m|}(\bx)\right), \qquad \bx \in B,\\[7mm]
\,\,\DST D_0(\bx_0) \frac{\partial  \phi_{jk}^{l}}{\partial \bn }\Big|_+ -(D_0(\bx_0)+D_1(\bx))\frac{\partial  \phi_{jk}^{l}}{\partial \bn }\Big|_- =\delta_l^0\, D_1\left(\bx \right) \bx^k \bn \cdot \nabla \bx^j\\
\quad \DST +D_0(\bx_0)\sum_{|m|=1}^{l} \frac{l!\partial^m D_0^{-1} (\bx_0)}{m! (l-|m|)!} D_1\left(\bx \right) \bx^m \frac{\partial  \phi_{jk}^{l-|m|}}{\partial \bn }\Big|_-, \qquad \bx\in\partial B,
\end{array} \right.
  \end{equation}
Here, $\bn$ is the outer normal to the boundary of $B$,
$\left.\frac{\partial \phi_{jk}^{l}}{\partial \bn }\right|_+$ (resp.
$\left.\frac{\partial \phi_{jk}^{l}}{\partial \bn }\right|_-$) denotes
the outer (resp. inner) trace of $\frac{\partial
  \phi_{jk}^{l}}{\partial \bn }$ on $\partial B$ as functions in
$H^{-\frac12}(\partial B)$.
\end{theorem}
The proof of the theorem is postponed to section \ref{sec:proofth}.

Theorem \ref{th2} has been proved in \cite{AK-04} by using single and
double layer potential techniques when the background diffusion
coefficient $D_0$ is constant on the entire domain $\Omega$ and when
$D_1$ is constant on $B$. Our result generalizes that of \cite{AK-04}
to the case of non-constant $D_0$ and $D_1$ for which layers
techniques are not available. The first order of the expansion can
also be obtained from the general formula proved in
\cite{CapVogM2AN1-02} and in \cite{CMV-IP}. 
\begin{remark} \rm
  The expansion in remark \ref{rem:alt} still holds for singular
  inclusions with $\Psi^\eps_j$ now the unique solution in
  $H^1_\textrm{loc}(\Rm^d) \cap \calC^\infty(\Rm^d \backslash
  \overline{B})$ to
\begin{displaymath}
  \begin{array}{rll}
    \Delta \Psi^\eps_j&=&0\qquad \bx \in \Rm^d/\overline{B},\\[3mm]
   \nabla\cdot\big( 1+D_1(\bx) D_{0}^{-1}(\bx_0+\eps\bx)\big)
  \nabla\Psi^\eps_j &=& - \nabla \cdot\big( D_1(\bx) D_{0}^{-1}(\bx_0+\eps\bx)
   \nabla \bx^j\big),\qquad \bx \in B,\\[3mm]
\Psi^\eps_{j}(\bx) &=& \calO(|\bx|^{1-d})\quad \textrm{as} \quad |\bx|\to \infty,
  \end{array}
\end{displaymath}
equipped with the jump condition on $\partial B$:
$$
\frac{\partial  \Psi^\eps_{j}}{\partial \bn }\Big|_+ -(1+D_1(\bx)D_0^{-1}(\bx_0+\eps \bx))\frac{\partial  \Psi^\eps_{j}}{\partial \bn }\Big|_- =D_1\left(\bx \right) D_0^{-1}(\bx_0+\eps \bx) \bn \cdot \nabla \bx^j, \qquad \bx\in\partial B.
$$
As in the end of remark \ref{rem:alt}, we could also derive a
modified asymptotic expansion in the case of singular inclusions.
\end{remark}
The above asymptotic expansions are compatible with the slightly
different expressions for the generalized polarization tensors
obtained in \cite{AK-04}.  We have the following proposition:
\begin{proposition} \label{prop:jump}
  Assume that $D_1$ is a non vanishing constant on $B$ and that $D_0$
  is constant on the set $\bx_0+\eps B$. Then $u^\eps$ verifies the
  following expansion, \textit{a.e.} on $\partial \Omega$,
\begin{eqnarray*}
\left. u^\eps(\by)\right|_{\partial \Omega}&=&\left.  U(\by)\right|_{\partial \Omega}-\sum_{|i|=1}^{d} \sum_{|j|=1}^{d}\frac{\eps^{d-2+|i|+|j|}}{i!j!}  \calM_{ij} \, \partial^j U (\bx_0)\, \left. \partial^i_\bx N(\bx_0,\by)\right|_{\partial \Omega}+\calO(\eps^{2d}),
\end{eqnarray*}
where $\calM$ is the generalized polarization tensor given in \cite{AK-04} by
\begin{eqnarray*}
\calM_{ij}&=& D_1 \int_{\partial B} \bn \cdot \nabla (\bx^j+\phi_{j}(\bx)) \bx^i d\sigma(\bx),\qquad i,j \in \NN^d.
\end{eqnarray*}
The functions $\phi_{j}$ are the unique solutions in
$H^1_{\textrm{loc}}(\Rm^d)\cap \calC^\infty((\Rm^d/\overline{B}) \cup
B)$ to the problem:
$$
 \left\{ 
\begin{array}{l}
 \DST \Delta \phi_{j}=0, \quad  \bx \in (\Rm^d/\overline{B}) \cup B,\\[3mm]
\DST D_0\left.\frac{\partial  \phi_{j}}{\partial \bn }\right|_+ -(D_0+D_1)\left.\frac{\partial  \phi_{j}}{\partial \bn }\right|_-=D_1\,\bn \cdot \nabla \bx^j, \quad \bx \in \partial B,\\
\DST \phi_j(\by)-\Gamma(\by) D_0^{-1}D_1\int_{\partial B} \bn \cdot \nabla \bx^j d\sigma(\bx)=\calO(|\by|^{1-d}), \quad \textrm{when } |\by|\to \infty.
\end{array} \right.
$$
\end{proposition}
The proof of the proposition is also postponed to section
\ref{sec:proofth}.
\subsection{Properties of the polarization tensor M} 
\label{subsec:prop}
In this section, we give some symmetry properties and estimates
satisfied by the tensors $M$ in theorems \ref{th:asymp} and \ref{th2}:
\begin{proposition}  \label{prop:M}
  Let $\alpha_i, \beta_i \in \Rm$, where $i$ belongs to a set a of
  multi-index $I$. Then, the polarization tensor $M$ verifies the
  following properties:
$$
\begin{array}{l}
 \DST (i) \qquad \sum_{i,j \in I} \alpha_i \beta_j M_{ij}= \sum_{i,j \in I} \alpha_i \beta_j M_{ji},\\
\DST (ii) \quad \int_B \frac{D_0(\bx_0) D_1(\bx)}{D_0(\bx_0)+D_1(\bx)} \Big| \nabla \Big( \sum_{i\in I} \alpha_i \bx^i \Big)\Big|^2 d\bx  \leq \sum_{i,j \in I} \alpha_i \alpha_j M_{ij} \leq \int_B D_1(\bx) \Big| \nabla \Big( \sum_{i\in I} \alpha_i \bx^i \Big)\Big|^2 d\bx.
\end{array}
$$
\end{proposition}
\begin{proof} Using the definition of $M$, we have,
\begin{eqnarray*}
 \sum_{i,j \in I}  \alpha_i \beta_j M_{ij}&=& \int_{B} D_1(\bx)  \nabla  \Big(\sum_{j \in I} \beta_j (\bx^j+\phi_{j\, 0}^0(\bx))\Big) \cdot \nabla  \Big(\sum_{i\in I}  \alpha_i \bx^i\Big) d\bx,\\
\end{eqnarray*}
and the system solved by $\phi_{j\, 0}^0$ deduced from
(\ref{eq:phijkl2}) imply that,
\begin{equation} \label{eqphi}
 \int_{\Rm^d }\Big(D_0(\bx_0)+D_1(\bx) \Big)\nabla \phi_{j0}^{0} \cdot \nabla \phi_{i0}^{0}\, d\bx=-
 \int_{B } D_1(\bx) \nabla \bx^j \cdot \nabla  \phi_{i0}^{0}\, d\bx.
\end{equation}
Consequently,
\begin{eqnarray*}
 &&\sum_{i,j \in I}  \alpha_i \beta_j M_{ij}= \int_{B} D_1(\bx)  \nabla  \Big(\sum_{j \in I}  \beta_j \bx^j \Big) \cdot \nabla  \Big(\sum_{i\in I}  \alpha_i \bx^i\Big) d\bx\\
&&\quad- \int_{\Rm^d }\Big(D_0(\bx_0)+D_1(\bx) \Big) \nabla \Big(\sum_{j \in I}  \beta_j \phi_{j0}^{0} \Big)\cdot \nabla \Big(\sum_{i \in I}  \alpha_i \phi_{i0}^{0} \Big)\, d\bx
 =\sum_{i,j \in I}  \alpha_i \beta_j M_{ji}.
\end{eqnarray*}
Concerning item $(ii)$, we remark from the above equality that:
$$
\sum_{i,j \in I}  \alpha_i \alpha_j M_{ij} \leq \int_{B} D_1(\bx)  \Big|\nabla  \Big(\sum_{j \in I} \alpha_j \bx^j)\Big) \Big|^2d\bx.
$$
For the other inequality, we split the sum as:
$$
 \sum_{i,j \in I}  \alpha_i \alpha_j M_{ij}=\int_{B} D_1\Big|\nabla  \Big(\sum_{j \in I} \alpha_j \bx^j)\Big) \Big|^2d\bx+\int_{B} D_1 \nabla  \Big(\sum_{j \in I}  \alpha_j \phi_{j0}^{0} \Big) \cdot \nabla  \Big(\sum_{i\in I}  \alpha_i \bx^i\Big) d\bx.
 $$
 Since $D_0(\bx_0)+D_1(\bx)$ is strictly positive \textit{a.e.} in
 $\Omega$, the Cauchy-Schwarz inequality yields
\begin{eqnarray*}
&&\int_{B} D_1  \nabla  \Big(\sum_{j \in I}  \alpha_j \phi_{j0}^{0} \Big) \cdot \nabla  \Big(\sum_{i\in I}  \alpha_i \bx^i\Big) d\bx\\
&\leq& \Big( \int_{B} (D_0(\bx_0)+D_1)\Big|\nabla  \Big(\sum_{j \in I}  \alpha_j \phi_{j0}^{0}\Big) \Big|^2d\bx\Big)^{\frac{1}{2}}
 \Big( \int_{B} \frac{D^2_1}{D_0(\bx_0)+D_1}\Big|\nabla  \Big(\sum_{i \in I}  \alpha_i \bx^i \Big) \Big|^2d\bx\Big)^{\frac{1}{2}}.
\end{eqnarray*}
In the same way, equation (\ref{eqphi}) gives:
$$\Big( \int_{B} (D_0(\bx_0)+D_1)\Big|\nabla  \Big(\sum_{j \in I}  \alpha_j \phi_{j0}^{0}\Big) \Big|^2d\bx\Big)^{\frac12} \leq  \Big( \int_{B} \frac{D^2_1}{D_0(\bx_0)+D_1}\Big|\nabla  \Big(\sum_{i \in I}  \alpha_i \bx^i \Big) \Big|^2d\bx\Big)^{\frac{1}{2}},$$
so that  
\begin{eqnarray*}
 \sum_{i,j \in I}  \alpha_i \beta_j M_{ij}&\geq& \int_{B} D_1  \Big|\nabla  \Big(\sum_{j \in I} \alpha_j \bx^j)\Big) \Big|^2d\bx-\int_{B} \frac{D^2_1}{D_0(\bx_0)+D_1}\Big|\nabla  \Big(\sum_{i \in I}  \alpha_i \bx^i \Big) \Big|^2d\bx,\\
&=&\int_{B} \frac{D_0(\bx_0) D_1(\bx)}{D_0(\bx_0)+D_1(\bx)}\Big|\nabla  \Big(\sum_{i \in I}  \alpha_i \bx^i \Big) \Big|^2d\bx.
\end{eqnarray*}
This ends the proof.
\end{proof}
\bigskip

Item $(ii)$ of the proposition is very similar to the estimates
obtained at the first order in \cite{CapVogM2AN1-02}. Such estimates
can be applied to verify the definiteness or not of the polarization
tensor. In particular, it gives:
$$
|\alpha|^2 \int_B \frac{D_0(\bx_0) D_1(\bx)}{D_0(\bx_0)+D_1(\bx)} d\bx  \leq \sum_{|i|=1,|j|=1} \alpha_i \alpha_j M_{ij} \leq |\alpha|^2 \int_B D_1(\bx)  d\bx,
$$
so that for $D_1$ constant, $M$ is positive definite when $D_1>0$
and negative definite when $D_1<0$, as it was shown in
\cite{AK-04,CMV-IP}. The only possibility to cancel the above sum is
then to set $D_1=0$, which means that there is no inclusion.
Therefore, an inhomogeneity with constant diffusion coefficient
always generates a perturbation of order $\eps^d$ on the measurements.
The situation is different when $D_1$ is not constant. Indeed, when
$\int_B D_1(\bx) d\bx<0$, then $M$ is negative definite, and when $\int_B
\frac{D_1(\bx)}{D_0(\bx_0)+D_1(\bx)} d\bx>0$, then $M$ is positive
definite. But when  $\int_B D_1(\bx) d\bx>0$ while at the same time
$\int_B \frac{D_1(\bx)}{D_0(\bx_0)+D_1(\bx)} d\bx<0$, then $M$ might
not be definite for a suitable choice of $D_1$ as we now show as
an application of the intermediate value theorem.
We show first that the functional $M_{ij}: L^\infty(\Omega) \to \Rm$,
$D_1 \to M_{ij}[D_1]$ is continuous.
\begin{lemma}
  There exists a positive constant $C$, such that, for all finite
  multi-index $i$ and $j$, we have:
$$
|M_{ij}[D_1^1]-M_{ij}[D_1^2]| \leq C  \| D_1^1-D_1^2\|_{L^\infty(B)}.
$$
\end{lemma}
\begin{proof} 
  Take two perturbation $D_1^1$, $D_1^2$ in $L^\infty(\Omega)$ with
  support in $B$ and denote by $M[D_1^1]$, $M[D_1^2]$ the
  corresponding polarization tensors. Then:
\begin{eqnarray}
M_{ij}[D_1^1]-M_{ij}[D_1^2]&=& \int_{B} (D^1_1-D^2_1)  \nabla  \bx^j\cdot \nabla \bx^i d\bx+\int_{B} (D^1_1-D^2_1)  \nabla  \phi_{j0}^0[D_1^1]\cdot \nabla \bx^i d\bx \nonumber\\
&&+\int_{B} D^2_1  \nabla  \Big(\phi_{j0}^0[D_1^1]-\phi_{j0}^0[D_1^2]\Big)\cdot \nabla \bx^i d\bx. \label{difM}
\end{eqnarray}
Introducing $w_j:=\phi_{j0}^0[D_1^1]-\phi_{j0}^0[D_1^2]$ and using the
equations verified by $\phi_{j0}^0[D_1^1]$ and $\phi_{j0}^0[D_1^2]$,
we find the relation:
\begin{eqnarray*}
\int_{\Rm^d}(D_0(\bx_0)+D_1^1) \big| \nabla w_j\big|^2 d\bx&=&-\int_B (D_1^1-D_1^2) \nabla w_j \cdot \big(\nabla \bx^j+\nabla  \phi_{j0}^0[D_1^2]\big) d\bx.
\end{eqnarray*}
Since $\nabla \phi_{j0}^0$ is bounded in $L^2(\Rm^d)$, this yields the
estimate
$$
\|\nabla w_j\|_{L^2(\Rm^d)} \leq C \| D_1^1-D_1^2\|_{L^\infty(B)}.
$$
Using (\ref{difM}), we obtain the desired result.
\end{proof}
\begin{lemma} 
  \label{lem:vanish}
  There exists a perturbation $D_1 \in L^\infty(\Omega)$ with $\int_B
  D_1(\bx) d\bx \neq 0$, such that, for a given $1\leq l\leq d$, the
  component $M_{e_l,e_l}[D_1]$ of the polarization tensor $M$
  vanishes, where $e_l$ is the $l$-th vector of the canonical basis of
  $\Rm^d$.
\end{lemma}
\begin{proof} Setting $\alpha_i=\delta_i^{e_l}$ in item $(ii)$ of proposition 
  \ref{prop:M} leads to
  $$
  \int_B \frac{D_0(\bx_0) D_1(\bx)}{D_0(\bx_0)+D_1(\bx)} d\bx \leq
  M_{e_l,e_l} \leq \int_B D_1(\bx) d\bx.
  $$
  Now take a $D_1^1$ such that $\int_B D_1^1(\bx) d\bx<0$.  
  Therefore, $M_{e_l,e_l}[D_1^1]<0$. We then continuously transform
  $D_1^1$ into  $D_1^2$ such that $\int_B \frac{D_0(\bx_0)
    D^2_1(\bx)}{D_0(\bx_0)+D^2_1(\bx)} d\bx >0$ keeping $\int_B D_1^1
  d\bx$ non zero in the transformation. Such a transformation exists:
  let indeed $D_1^1$ be a bounded function in $B$ with positive and
  negative parts $D_+^1$ and $D_-^1$. We set $\int_B D_-^1 d\bx>\int_B
  D_+^1 d\bx$ so that $\int_B D_1^1(\bx) d\bx<0$. Letting the negative
  part $D_-^1$ continuously go to zero then gives a possible
  transformation. For the resulting $D_1^2$, we have
  $M_{e_l,e_l}[D_1^2]>0$. Since the functional $M_{e_l,e_l}[D_1]$ is
  continuous from $L^\infty(B)$ to $\Rm$, we deduce from the
  intermediate value theorem the existence of a $D_1^*$ with $\int_B
  D_1^* d\bx \neq 0$ such that $M_{e_l,e_l}[D_1^*]=0$. This ends the
  proof of the proposition.
\end{proof}
As a corollary of the previous result, we have
\begin{proposition} 
  \label{prop:vanish}
  There exists a perturbation $0\not\equiv D_1 \in L^\infty(\Omega)$ with
  spherical symmetry such that $M_{ij}\equiv0$. 
\end{proposition}
\begin{proof}
  Consider an inclusion with spherical symmetry. We find that
  $M_{ij}=M_0 \delta_i^j$ when $|i|=|j|=1$ so that the above lemma
  yields the existence of non-vanishing perturbation such that $M_0=0$
  and consequently no term of order $\eps^d$ appears in the asymptotic
  expansion.
\end{proof}
The latter result is to be compared with the case where $D_1$ is
constant for which there is always a contribution of order $\eps^d$ in
the expansion provided the constant is not zero.

\section{Perturbations in  the Helmholtz equation}
\label{sec:helm}
This section addresses the problem of small-volume inhomogeneities in
the Helmholtz equation. As we did for the diffusion equation, we
derive an asymptotic expansion of the perturbed solution in the volume
of the inclusions.

\subsection{Asymptotic expansion and polarization tensors}
\label{sec:asexphelm}
We consider the following Helmholtz (or Schr\"odinger) equation
posed in a bounded Lipschitz domain $\Omega$ of $\Rm^d$, $d \geq 2$,
and with $d\leq5$ for technical reasons:
\begin{equation}
\label{eq:v}
 \left\{ 
\begin{array}{l}
 \DST-\Delta v^\eps(\bx)+\Big(q_0(\bx)+\frac{1}{\eps^{2-\eta}} q_1\Big(\frac{\bx-\bx_0}{\eps} \Big) \Big)\, v^\eps(\bx)=0, \quad  \bx \in \Omega,\\[2mm]
\DST  \frac{\partial v^\eps}{\partial \bn } =g \in L^2(\partial \Omega)\quad \textrm{on } \partial \Omega,
\end{array} \right.    
\end{equation}      
where $\bx_0$ is a given point in $\Omega$, $q_0 \in L^\infty(\Omega)$
is the background index or potential, and $q_1 \in L^\infty(\Omega)$
is a local perturbation, with support localized in a bounded Lipschitz
domain $B$. We consider the case with only one inclusion, knowing that
the results below generalize to the setting with several
well-separated inclusions so long as the maximal order in the
expansion is sufficiently small so that the inclusions do not interact
at that order. The perturbation has a magnitude of order
$\eps^{\eta-2}$, with $\eta \in [0,2]$. The most interesting case is
$\eta=0$, which corresponds to the strongest type of perturbation. The
latter case allows to relate the asymptotic formula given in the
preceding section to the one that we propose below for a particular
form of the potential $q_1$.

When $q_0$ is negative, the above system models waves propagating in a
medium perturbed by a small inclusion of diameter $\eps $ with a
refractive index of order $\eps^{\eta-2}$. We refer to
\cite{HV-preprint} and \cite{HPV-preprint} for the case of
high-frequency waves in dimension two perturbed by small inclusions
with index of order one. The case $q_0$ and $q_1$ constant with $q_0$
negative and $\eta=2$ has been treated in \cite{AK-04} with Dirichlet
conditions instead of Neumann conditions at the domain's boundary.
When $q_0$ is positive, (\ref{eq:v}) models e.g. diffusive light
propagating in a medium with background absorption $q_0$ and zones of
different absorption coefficients in a small volume. The case $\eta=2$
has been investigated in dimension three in \cite{SmallAbso-02} for a
constant background $q_0$ and a constant perturbation $q_1$.

We denote by $V$ the solution of the unperturbed equation
\begin{equation} 
\label{eq:V}
 \left\{     
\begin{array}{l}
 \DST-\Delta V+q_0 V=0, \quad  \bx \in \Omega,\\[2mm]
\DST  \frac{\partial V}{\partial \bn } =g \quad \textrm{on } \partial \Omega.
\end{array} \right.
\end{equation}
When $q_0 \equiv 0$, we assume the normalizing and compatibility
conditions:
\begin{equation} \label{eq:norma} \int_{\partial \Omega} V d\sigma=0 \qquad \mbox{ and } \qquad  \int_{\partial \Omega} g d\sigma=0,\end{equation}
where $\sigma$ denotes the surface measure on $\partial \Omega$.
According to (\ref{eq:v}), this also implies:
\begin{equation} \label{eq:meanzero}
\int_{\Omega}q_1\Big(\frac{\bx-\bx_0}{\eps} \Big)\, v^\eps(\bx) d\bx=0,\qquad \textrm{when } q_0=0.
\end{equation}
In order to obtain the existence and uniqueness of a variational
solution to
$(\ref{eq:V})$, we make the following classical assumption: \\
\\
\noindent \textbf{(H-1)}\,\, Let $u\in H^1(\Omega)$. Then
$$
\int_\Omega \nabla u \cdot \nabla v \, d\bx +\int_\Omega q_0 \,u \, v d\bx=0, \qquad \mbox{ for all } v\in H^1(\Omega),
$$
implies that $u=0$.

Under \textbf{(H-1)}, an application of lemma \ref{lemappend} of the
appendix yields a unique weak solution $V \in H^1(\Omega)$ to
$(\ref{eq:V})$. When $q_0:=0$, the same holds thanks to conditions
(\ref{eq:norma}). Since we need high-order Taylor expansions of $V$ in
the sequel, we make the additional assumption that the restriction of
$q_0$ to a neighborhood $\bx_0+\eps B'$ of the set $\bx_0+\eps B$,
with $B \subset \subset B'$, belongs to $\calC^\infty(\bx_0+\eps B')$.
Using standard elliptic regularity \cite{gt1} and (\ref{hypB}), we
obtain that $V \in \calC^\infty(\bx_0+\eps B')$.  When first order
expansions are considered, then a $L^\infty(\Omega)$ regularity for $V$
is sufficient.  Existence and uniqueness for \eqref{eq:v} uniformly in
$\eps$ for $\eps$ small enough will be given in the sequel. When $\eta
\in ]0,2]$, no additional condition is required on $q_1$. When
$\eta=0$, we add the following assumption:

\noindent \textbf{(H-2)}\,\,
$-1$ is not an eigenvalue of the bounded operator $T$ defined as:
$$
T: L^2(B) \to L^2(B), \quad \varphi \to T\varphi(\by)=\int_{B}
q_1(\bx) \varphi(\bx) \Gamma(\bx-\by) d\bx.
$$
Here, $\Gamma$ is the fundamental solution of the Laplacian given
in (\ref{eq:G}). $\textbf{(H-2)}$ is verified for instance when
$q_1>0$ \textit{a.e.} in $B$ or when the following Rollnick type
\cite{RS-80-4} norm of $q_1$ is less than one,
$$
\int_{B}\int_{B} \Big(\sqrt{|q_1(\bx)|}\sqrt{|q_1(\by)|} |\Gamma(\bx,\by)| \Big)^p d\bx d\by <1,
$$
for some $p \geq 1$, or when $q_1$ is a Bohm-like potential of the
form
$$
q_1(\bx)=\frac{\Delta \sqrt{1+D_1(\bx)}}{\sqrt{1+D_1(\bx)}}, 
$$
for some $\calC^2(\Rm^d)$ function $D_1$ with support in $B$ such
that $1+D_1>0$ in $\Rm^d$.

The case $d=2$ and $\eta=0$ is particular in the sense that
$$
\frac{1}{\eps^2}\int_{\Omega}q_1\Big(\frac{\bx-\bx_0}{\eps} \Big)\, d\bx=\int_B q_1(\bx) d\bx=\calO(1),
$$
so that we cannot expect the perturbation caused by the inclusion
to be small in the general case. We thus need to add an additional
hypothesis to be able to treat $q_1$ as a perturbation. It is the case
under the following symmetry assumption:
\\
\noindent \textbf{(H-3)} When $d=2$ and $\eta=0$, 
we assume that the solution $v^\eps$ to ($\ref{eq:v}$) verifies that
$$
\int_{\Omega}q_1\Big(\frac{\bx-\bx_0}{\eps} \Big)\, v^\eps(\bx)
d\bx=0.$$
\\

Note that $\textbf{(H-3)}$ is verified when e.g. $q_0 \equiv 0$ thanks
to (\ref{eq:norma}). We introduce the Green function $N(\bx,\by) \in
\calD'(\Omega \times \Omega)$ of
\eqref{eq:V}, which for each fixed $\by$ in
$\Omega$, solves:
\begin{equation}
\label{eq:Nq}
 \left\{ 
\begin{array}{l}
 \DST-\Delta_\bx N(\bx,\by)+q_0(\bx) N(\bx,\by)= \delta(\bx-\by), \quad  \bx \in \Omega,\\[2mm]
\DST  \frac{\partial N(\bx,\by)}{\partial \bn_\bx } =0 \quad \textrm{on } \partial \Omega.
\end{array} \right. 
\end{equation}
When $q_0 \equiv 0$, $N$ has to be defined as in (\ref{eq:N}). $N$ is
symmetric in its arguments. Hypothesis \textbf{(H-1)} is verified e.g.
when $q_0 \geq 0$, $\Omega$ \textit{a.e.} (with the normalizing
condition when $q_0\equiv 0$), when $q_0$ is constant and not an
eigenvalue of the Laplacian equipped with homogeneous Neumann
conditions, or when the following Rollnick-type norm of $q_0$ is less
than one,
$$
\int_{\Omega}\int_{\Omega} \Big(\sqrt{|q_0(\bx)|}\sqrt{|q_0(\by)|} |N(\bx,\by)| \Big)^p d\bx d\by <1,
$$
for some $p\geq1$.  We have the following proposition,
which allows us to decompose $N$ as the sum of the whole space
Green function $\Gamma$ and a regular function:
\begin{proposition} \label{prop:decompN}
  We have $N(\bx,\by):=\Gamma(\bx-\by)+R(\bx,\by)$, where
  $R(\cdot,\by) \in H^1(\Omega)\cap W^{2,p}(\Omega')$ with $p <
  \frac{d}{d-2}$ when $3\leq d\leq 5$ and $p <\infty$ when $d=2$ for
  any $\Omega' \subset \subset \Omega$ uniformly in $\by\in \Omega'$.
  When $q_0\equiv 0$, then $R$ belongs to $\calC^\infty(\Omega \times
  \Omega)$. Moreover, $N$ admits the following asymptotic expansion
  for $\bx \in B$, $\by$ \textit{a.e.}  in $\partial \Omega$:
\begin{equation} \label{decompNhelm}
\nabla _\bx N(\bx_0+\eps \bx,\by)=\sum_{|i|=1}^{d} \frac{\eps^{|i|}}{i!}  \nabla \bx^i \partial^i_\bx N(\bx_0,\by)+\calO(\eps^{d+1}),
\end{equation} 
where $\calO(\eps^{d+1})$ denotes a term bounded in $L^2(\partial
\Omega)$ by $C\eps^{d+1}$, uniformly in $\bx$.
\end{proposition}
\begin{proof}
  We consider only the case $q_0 \neq 0$ since the case $q_0\equiv 0$
  follows from proposition \ref{prop:decompNdiff}. Plugging
  $N(\bx,\by):=\Gamma(\bx-\by)+R(\bx,\by)$ into (\ref{eq:Nq}) leads
  for any $\by$ fixed in $\Omega$ to the equation:
\begin{equation}
\label{eq:R1}
 \left\{  
\begin{array}{l}
 \DST-\Delta_\bx R(\bx,\by)+q_0(\bx)R(\bx,\by)=-q_0(\bx)\Gamma(\bx-\by), \quad  \bx \in \Omega,\\[2mm]
\DST  \frac{\partial R(\bx,\by)}{\partial \bn_\bx } = -\frac{\partial \Gamma(\bx-\by)}{\partial \bn_\bx },\quad \textrm{on } \partial \Omega.
\end{array} \right. 
\end{equation} 
Pick an $\by \in \Omega' \subset \subset \Omega$ and for any $v \in
H^1(\Omega)$, consider the linear form:
$$l(v):=- \int_\Omega q_0(\bx)\Gamma (\bx-\by)v(\bx)
d\bx-\int_{\partial \Omega}\frac{\partial \Gamma(\bx-\by)}{\partial
  \bn_\bx } v(\bx) d\sigma(\bx). 
$$
Then $l$ is continuous in $H^1(\Omega)$. Indeed, on the one hand,
$\Gamma(\bx-\by)$ is uniformly bounded for $(\bx,\by) \in \partial
\Omega \times \Omega'$ which allows us to treat the second integral.
On the other hand, $ \Gamma \in L^{p}_\textrm{loc}(\Rm^d)$ with $p <
\frac{d}{d-2}$ when $d\geq 3$ and $p <\infty$ when $d=2$ so that the
Sobolev embedding $H^1(\Omega) \hookrightarrow L^q(\Omega)$, for $q
\leq \frac{2d}{d-2}$ when $d\geq 3$ and $q<\infty$ when $d=2$ implies
$$
|l(v)| \leq C (\|\Gamma \|_{L^{q'}(B_R)}+1) \|v\|_{H^1(\Omega)},
$$
for $q'\geq \frac{2d}{d+2}$ when $d\geq 3$ and $q'>1$ when $d=2$,
where $B_R$ is a ball of radius $R$ large enough. Since
$\frac{d}{d-2}>\frac{2d}{d+2}$ for $d<6$, we get the desired result.
Note that for $d\geq7$, the above linear form is not continuous as we may
construct functions $v\in H^1(\Omega)$ of the form $|\bx|^{-\alpha}$
such that $\Gamma(\bx)v(\bx)$ is not integrable in the vicinity of
$0$.  Lemma \ref{lemappend} then yields a unique $R(\cdot,\by) \in
H^1(\Omega)$ uniformly bounded in $\by$ when $\by \in \Omega'$ by
choosing $a_0(u,v)=\int_{\Omega}(\nabla u\cdot\nabla v + uv) d\bx$ and
$a_1(u,v)=\int_{\Omega} (q_0(\bx)-1)uv d\bx$.  Standard elliptic
regularity \cite{gt1} gives, for $1<p<\frac{d}{d-2}$ when $d\geq 3$ and
$p <\infty$ when $d=2$, that:
\begin{eqnarray*}
\| R(\cdot,\by)\|_{W^{2,p}(\Omega')} &\leq& C \left( \| R(\cdot,\by)\|_{H^1(\Omega)} +  \| \Gamma \|_{L^p(B_R)}\right),
\end{eqnarray*}
so that $R(\cdot,\by) \in W^{2,p}(\Omega')$ uniformly in $\by \in \Omega'$.

To prove (\ref{decompNhelm}), we decompose $R$ as
$R(\bx,\by):=R_1(\bx,\by)+R_2(\bx,\by)$ with
\begin{equation}
\label{eq:RR1}
 \left\{  
\begin{array}{l}
 \DST-\Delta_\bx R_1(\bx,\by)+q_0(\bx)R_1(\bx,\by)=-q_0(\bx)\Gamma(\bx-\by), \quad  \bx \in \Omega,\\[2mm]
\DST  \frac{\partial R_1(\bx,\by)}{\partial \bn_\bx } =0,\quad \textrm{on } \partial \Omega,
\end{array} \right. 
\end{equation} 
\begin{equation}
\label{eq:RR2}
 \left\{  
\begin{array}{l}
 \DST-\Delta_\bx R_2(\bx,\by)+q_0(\bx)R_2(\bx,\by)=0, \quad  \bx \in \Omega,\\[2mm]
\DST  \frac{\partial R_2(\bx,\by)}{\partial \bn_\bx } =-\frac{\partial \Gamma(\bx-\by)}{\partial \bn_\bx },\quad \textrm{on } \partial \Omega.
\end{array} \right. 
\end{equation} 
Consider first (\ref{eq:RR1}) for $\by \in \partial \Omega$. According
to lemma \ref{lemappend}, $R_1(\cdot,\by)$ belongs to $H^1(\Omega)$
and is uniformly bounded with respect to $\by$. Let $B'$ be a
neighborhood of $B$ such that $B \subset \subset B'$. Since
$\Gamma(\cdot-\by) \in \calC^\infty(\bx_0+\eps \overline{B'})$
uniformly in $\by \in \partial \Omega$, and $q_0 \in
\calC^\infty(\bx_0+\eps B')$, we obtain from elliptic regularity that
$R_1(\cdot,\by) \in \calC^\infty(\bx_0+\eps \overline{B})$ uniformly
in $\by \in \partial \Omega$.  Now, $R_2$ is treated almost exactly as
the term $R_2$ in proposition \ref{prop:decompNdiff}, so we highlight
the differences. According to the previous results on $R$, the trace
$\left.  N(\bx,\bz) \right|_{\partial \Omega}$ exists in $L^2(\partial
\Omega)$ uniformly for $\bz \in \Omega' \subset \subset \Omega$.  Thus
we have the following integral equation:
$$
R_2(\bz,\by)=- \int_{\partial \Omega}\frac{\partial  
  \Gamma(\bx-\by)}{\partial \bn_\bx } N(\bx,\bz)d\sigma(\bx), 
  \qquad (\bz,\by) \in \Omega' \times \Omega.
$$
As $\by$ goes to $\partial \Omega$, the integral converges to
$$
- \textrm{p.v} \int_{\partial \Omega} \frac{\partial
  \Gamma(\bx-\by)}{\partial \bn_\bx }N(\bx,\bz) d\sigma(\bx)
+\frac{1}{2} N(\by,\bz), $$
where p.v. stands for the Cauchy principal
value and the above quantity makes sense in $L^2(\partial \Omega)$
uniformly in $\bz \in \Omega'$ so that $R_2(\bz,\cdot) \in
L^2(\partial \Omega)$ for all $\bz \in \Omega'$. Moreover, we verify
that $R_2(\bz,\by)$ satisfies in the distributional sense, for $\bz
\in \Omega'$, $\by \in \partial \Omega$,
$$-\Delta_\bz R_2(\bz,\by)+q_0(\bz)R_2(\bz,\by)=0,$$
so that we
conclude from elliptic regularity that $R_2(\cdot,\by) \in
\calC^\infty(\bx_0+\eps \overline{B})$ with values in $L^2(\partial
\Omega)$. Classical Taylor expansions then yield
$(\ref{decompNhelm})$.
\end{proof}
We come back to \eqref{eq:v} and state the following result.
\begin{proposition} \label{prop:asympV} 
  Assume that \textbf{(H-2)} is satisfied when $\eta=0$ and
  \textbf{(H-3)} is satisfied when $d=2$ and $\eta=0$. Then, under assumption
  \textbf{(H-1)}, there exists $\eps_0>0$, such that for all $0<\eps <
  \eps_0$, the system (\ref{eq:v}) admits a unique variational
  solution $v^\eps \in H^1(\Omega)$.  Moreover, the restriction of
  $v^\eps$ to the set $\bx_0+\eps B$ verifies the following
  decomposition
\begin{equation} \label{decompvtheo}
v^\eps(\bx_0+\eps \by)=V(\bx_0+\eps \by)+\eps^\eta \Psi^\eps(\by)+\eps^{d-2+\eta}\,r^\eps(\by)+\calO(\eps^{d+2}), \qquad \by \textit{ a.e.} \textrm{ in } B, 
\end{equation}
where $\Psi^\eps(\by):=\sum_{|j|=0}^{d+1}\frac{\eps^{|j|}}{j!}
\partial^j V(\bx_0) \phi^\eta_j(\by)$ and $\phi^\eta_j$ is the unique
solution in $H^1(B)$ to
\begin{equation} \label{eqphith2}
\phi^\eta_j+\eps^{\eta} T \phi^\eta_j = -T \bx^j, \qquad \by \in B,
\end{equation}
and $r^\eps$ the unique solution in $H^1(B)$, for $\by \in B$, to
$$
r^\eps(\by)+\eps^{\eta} T r^\eps(\by) = \int_{B} q_1\left(\bx \right) v^\eps(\bx_0+\eps \bx) \left(R(\bx_0+\eps \bx, \bx_0+\eps \by) -\delta_d^2 \,(2\pi)^{-1}\, \log \eps \right)d\bx.
$$
The operator $T$ is defined in \textbf{(H-2)} and the function $R$
in proposition \ref{prop:decompN} whereas $\delta_d^2$ is the
Kronecker symbol. The notation $\calO(\eps^{d+2})$ represents a term
bounded in $H^1(B)$ by $C\eps^{d+2}$.  The remainder $r^\eps$ is
bounded in $L^2(B)$ independently of $\eps$ when $d =3$, by $C
\eps^{-\alpha}$, for any $\alpha>0$ when $d=4$, by $C\eps^{-1}$ when
$d=5$, and by $ C |\log \eps|$ when $d=2$.  When $d=2$ and $\eta=0$,
then $r^\eps$ is of order $\calO(\eps)$ thanks to \textbf{(H-3)}. When
$q_0\equiv 0$, then $r^\eps$ is bounded in $H^1(B)$ independently of
$\eps$ for any $d$.
\end{proposition}
We then have the following theorem:
\begin{theorem} \label{th:asympV}
  Under the hypotheses of proposition \ref{prop:asympV}, the solution
  to $v^\eps$ to (\ref{eq:v}) satisfies the following asymptotic
  expansion, almost everywhere on $\partial \Omega$:
\begin{eqnarray*}
\left. v^\eps(\by) \right|_{\partial \Omega}&=&\left. V(\by)\right|_{\partial \Omega}-\sum_{|j|=0}^{d+1} \sum_{|i|=0}^{d+1} \frac{\eps^{d-2+\eta+|i|+|j|}}{i!j!}  \left(Q_{ij}+\eps^\eta Q^\eta_{ij} \right) \partial^j V (\bx_0) \left. \partial^i N(\bx_0,\by)\right|_{\partial \Omega}\\
&&+\eps^{2(d-2+\eta)} f^\eps(\by)+\calO(\eps^{2d}),
\end{eqnarray*}
where $\calO(\eps^{2d})$ is a term bounded in $L^2(\partial \Omega)$
by $C\eps^{2d}$ and for $(i,j)\in\NN^d \times \NN^d$,
\begin{displaymath}
  \begin{array}{rcl}
\DST Q_{ij}&=&\DST\int_B q_1(\bx)\bx^j \bx^i d\bx, \qquad Q^\eta_{ij}\,\,=\,\,\int_B q_1(\bx) \phi_j^\eta(\bx) \bx^i d\bx, \\[3mm]
\DST  f^\eps(\by)&=&\DST \int_B q_1(\bx) r^\eps (\bx) N(\bx_0+\eps \bx,
  \by) d\bx.
  \end{array}
\end{displaymath}
The remainder $\| f^\eps \|_{L^2(\partial \Omega)}$ is of order:
$\calO(|\log \eps|)$
when $d=2$; $\calO(1)$ when $d=3$; $\calO(\eps^{-\alpha})$ for
any $\alpha>0$ when $d=4$; and $\calO(\eps^{-1})$ when $d=5$.
\end{theorem}
The proofs of the proposition and the theorem are given in section
\ref{proofprop}. When $\eta>0$, $\phi_j^\eta$ still depends on $\eps$.
We may then expand the operator $(\calI+\eps^\eta T)^{-1}$ in terms of
Neumann series up to the right order. We include the term $f^\eps$ in
the formula because we need its explicit expression below to make the
link between the asymptotic expansion for the diffusion equation and
that for the Helmholtz equation.

In the particular case where $q_0$ constant and positive, $\eta=2$,
$q_1$ is constant, and the inclusion is centered at $\bx_0$ so
that
\begin{math}
\int_B \bx d\bx=0,
\end{math}
we find for $d=3$ that
\begin{eqnarray*}
v^\eps(\by)&=&V(\by)- \eps^3\,q_1\, \left( \int_B \left( 1+\eps^2 \,\phi^2_0\right) d\bx \right)V(\bx_0)N(\bx_0,\by)\\
&&-\,q_1\sum_{|j|=0}^2 \sum_{|i|+|j|=2} \frac{\eps^{5}}{i!j!}\left(\int_B \bx^i \bx^j d\bx \right)\partial^i N(\bx_0,\by) \partial^j V (\bx_0)+\calO(\eps^6).
\end{eqnarray*}
According to (\ref{eqphith2}), $\phi^2_0$ verifies
$\phi^2_0=-T1+\calO(\eps^2)$ so that we recover the asymptotic
expansion given in \cite{SmallAbso-02}.

The tensor $Q$ is clearly symmetric. When $q_1$ is constant and not
identically zero, there is always a contribution of order
$\eps^{d-2+\eta}$ in the expansion, while for spatially varying $q_1$,
the first order contribution can vanish for instance by choosing $q_1$
such that $\int_B q_1 d\bx=0$.

\subsection{Relation between the diffusion and Helmholtz equations} 
\label{sec:rel}
We now compare the asymptotic expansions for the solution $u^\eps$ to
the diffusion equation (\ref{eq:diff}) given in theorem \ref{th:asymp}
and for the solution $v^\eps$ to the Helmholtz equation (\ref{eq:v})
given in theorem \ref{th:asympV}. It is well-known that a solution to
the diffusion equation
$$
\nabla \cdot D \nabla u=0,
$$
with $D \in \calC^2(\Rm^d)$ for instance and strictly positive,
also satisfies a Helmholtz or Schr\"odinger equation of the form
\begin{displaymath}
  \Delta \big(\sqrt{D} u \big)+ \Big(\dfrac{\Delta \sqrt{D}}{\sqrt D}
   \Big)\big( \sqrt D u\big)=0.
\end{displaymath}
Our purpose here is to verify that the polarization tensors obtained
in the diffusion and Helmholtz frameworks are indeed the same for the
specific form of the potential $q_1$ that allows us to transform one
equation into the other. As in section \ref{sec:diff}, we define
$D^\eps(\bx)=D_0(\bx)+D_1(\frac{\bx-\bx_0}{\eps})$ and to simplify the
presentation, assume that $D_0$ is constant in $\Omega$.  We assume
that $D_1 \in \calC^2(\Omega)$ with support included in $B$  and that
$D_0+D_1$ is strictly positive in $\Omega$, so that we can define
\begin{equation} \label{eq:defq1}
q_1(\bx):=\frac{\Delta \sqrt{D_0+D_1(\bx)}}{\sqrt{D_0+D_1(\bx)}}.
\end{equation}
We then consider the function $v^\eps$ which satisfies (\ref{eq:v})
with $q_0 = 0$, $\eta=0$ and $q_1$ defined as above. With such a
choice, the quantity
$$
\frac{v^\eps(\bx)}{\sqrt{D_0+D_1(\frac{\bx-\bx_0}{\eps})}}
$$
solves (\ref{eq:diff}). Since $\eta=0$, we may expect from the
expansion given in theorem \ref{th:asympV} that the inclusion induces
a correction of order $\eps^{d-2}$ whereas the same inclusion induces
a correction of order $\eps^d$ in the diffusion equation.  Some
simplifications due to the particular form of the potential $q_1$ must
render the correction of order $\eps^d$ in the Helmholtz framework as
well. We state the main result of this section:
\begin{proposition} \label{prop:equiv} When $q_1$ has the form
  (\ref{eq:defq1}), then we have
\begin{eqnarray}
\sum_{j=0}^{d+1} \frac{\eps^{|j|}}{j!} \partial^j V (\bx_0)\left(Q_{0j}+Q_{0j}^0\right)&=&\calO(\eps^{d+2}), \label{eq:zero1}\\
\sum_{i=0}^{d+1} \frac{\eps^{|i|}}{i!} \partial^i N (\bx_0,\by)\left(Q_{i0}+Q_{i0}^0\right)&=&\calO(\eps^{d+2}).\label{eq:zero2}
\end{eqnarray}
Here, the index $0$ of the polarization tensors represents the vector
of $\NN^d$ with components all equal to zero. We have the following
relation between the polarization tensor $M$ in the context of theorem
\ref{th:asymp} and the polarization tensor  $\tilde
M:=\sqrt{D_0}(Q+Q^0)$ in the context of the Helmholtz equation:
\begin{equation}
M_{ij}=\tilde M_{ij}, \qquad |i|=|j|=1, \label{equiv1}
\end{equation}
\begin{equation}
\begin{array}{l}
\DST \sum_{|j|=1}^{d+1} \sum_{|i|=1}^{d+1} \frac{\eps^{|i|+|j|}}{i!j!} \partial^i N(\bx_0,\by) \partial^j V (\bx_0) \big(M_{ij}-
\tilde M_{ij}\big) =\calO(\eps^{d+2}).\label{equiv2}
\end{array}
\end{equation} 
\end{proposition}
\bigskip

The proof of the proposition is given in section \ref{proofprop}.
Equations (\ref{eq:zero1}) and (\ref{eq:zero2}) imply that the two
first orders in the expansion of theorem \ref{th:asympV} vanish so
that the correction is of order $\eps^d$. Equations (\ref{equiv1}) and
(\ref{equiv2}) show the equivalence of the tensors $M_{ij}$ and
$\tilde M_{ij}$ for $|i|,|j| \leq d+1$ up to an error of order
$\eps^{d+2}$, which is sufficient to show that the asymptotic
expansions on $u^\eps$ and $v^\eps$ agree up to the order $\eps^{2d}$.
The proofs can in fact be modified to show the equivalence at higher
orders as well, \textit{e.g.}, for any $r \in \NN$,
$$ 
\DST \sum_{|j|=1}^{r+1} \sum_{|i|=1}^{r+1} \frac{\eps^{|i|+|j|}}{i!j!} \partial^i N(\bx_0,\by) \partial^j V (\bx_0) \big(M_{ij}-
\tilde M_{ij}\big) =\calO(\eps^{r+2}).
$$
Furthermore, denoting by $(m_{ij})$ the modified polarization
tensor obtained from $\Phi_j$ at the end of remark \ref{rem:alt}, we
can show in this context the strict equality between the Helmholtz and
diffusion tensors, that is $\tilde M_{ij}=m_{ij}$, for all $i,j$.

\section{Proofs of the main results} 
\label{sec:proofth}

\subsection{Asymptotic expansions for the diffusion equation}
\label{sec:resdiff}

We now prove theorems \ref{th:asymp} and \ref{th2} and proposition
\ref{prop:jump}.

\begin{proofof}{\em Theorem \ref{th:asymp}}.
The starting point of the proof is the formulation of (\ref{eq:diff})
as the following integral equation:
\begin{eqnarray}
u^\eps(\by)&=&U(\by)-\int_{\bx_0+\eps B} D_1\left(\frac{\bx-\bx_0}{\eps}\right) \nabla u^\eps(\bx) \cdot 
\nabla_\bx N(\bx,\by) d\bx,\nonumber \\ \label{eq:int}
&=&U(\by)-\eps^d\int_{B} D_1\left(\bx\right) \nabla u^\eps(\bx_0+\eps \bx)
 \cdot \nabla_\bx N(\bx_0+\eps \bx,\by) d\bx.
\end{eqnarray} 
The above equation is justified rigorously as in the derivation of
(\ref{integappend}) in lemma \ref{lemappend1} of the appendix. We
highlight the main differences.  According to proposition
\ref{prop:decompNdiff}, we have $\nabla_\bx N(\bx,\by)=D_0^{-1}(\bx)
\nabla \Gamma(\bx-\by)+\nabla_\bx R_2(\bx,\by)$, with $\nabla_\bx
R_2(\cdot,\by) \in L^2(\Omega)$ for every $\by$ in $\Omega$ so that
the above equation makes sense in $L^2(\Omega)$ and therefore almost
everywhere in $\Omega$ thanks to the Young inequality since $\nabla
u^\eps \in L^2(\Omega)$ and $\nabla \Gamma \in
L^1_\textrm{loc}(\Rm^d)$. The integral equation (\ref{eq:int}) is
obtained from the variational formulations of (\ref{eq:diff}) and
(\ref{eq:diffD0}):
\begin{eqnarray}
\int_\Omega D^\eps \nabla u^\eps \cdot \nabla v\, d\bx &=& \int_{\partial \Omega} g v \,d\sigma(\bx) \,\,=\,\,\int_\Omega D_0 \nabla U \cdot \nabla v \,d\bx \label{varia-uU},
\end{eqnarray}
for all $v \in H^1(\Omega)$.  Then, let $\varphi \in L^2(\Omega)$ and
set $v(\bx):=\int_{\Omega} N(\bx,\by) \varphi(\by) d\by$. Thus $v$ is
the unique solution in $H^1(\Omega)$ to $ -\nabla \cdot D_0 \nabla v=
\varphi $ equipped with homogeneous Neumann conditions and the
normalization $\int_{\partial \Omega} v d\sigma(\bx)$=0. As in the
proof of (\ref{integappend}) or in the proof of proposition
\ref{prop:jump}, we verify that Fubini's theorem applies and that
$$
\int_{\Omega} \left( \int_{\Omega} D_0(\bx) \nabla u(\bx) \cdot \nabla_\bx N(\bx,\by) d\bx- u(\by) \right) \varphi(\by) d\by=0, \qquad \forall u \in H^1(\Omega).
$$
Applying the latter equality to both $u^\eps$ and $U$, gives
(\ref{eq:int}) together with (\ref{varia-uU}).

To continue the proof of theorem, we write $u^\eps=U+w^\eps$ as the
sum of the unperturbed solution $U$ and a corrector $w^\eps$, solution
of
\begin{equation}
\label{eq:diffw}
\begin{array}{rll}
 \nabla \cdot \left(D_0(\bx)+D_1\left(\frac{\bx-\bx_0}{\eps} \right) \right)\nabla w^\eps=-\nabla \cdot D_1\left(\frac{\bx-\bx_0}{\eps} \right) \nabla U, \quad \textrm{in}\; \Omega,\\[1mm]
\DST \frac{\partial w^\eps}{\partial \bn }=0, \quad \textrm{on}\; \partial \Omega,\quad
\DST \int_{\partial \Omega} w^\eps(\bx) d \sigma(\bx)=0.
\end{array} 
\end{equation} 
Since both $u^\eps$ and $U$ belong to $H^{1}(\Omega)$, then $w^\eps
\in H^{1}(\Omega)$ and we deduce from (\ref{eq:diffw}) that:
$$
\|\nabla w^\eps \|_{L^2(\Omega)} \leq C \eps^{\frac{d}{2}}
\|D_1\|_{L^\infty(B)} \|\nabla U\|_{L^\infty(B_0)},
$$
for some $\bx_0+\eps_0 B\subset B_0\subset\subset\Omega$ with
$\eps_0>0$ so that, from standard elliptic regularity,
\begin{equation} \label{estim:w}
\|\nabla  w^\eps(\bx_0 +\eps \cdot) \|_{L^2(B)} \leq C \|D_1\|_{L^\infty(B)}  \|\nabla U\|_{L^\infty(B_0)} \leq C \|D_1\|_{L^\infty(B)} \|g\|_{L^2(\partial \Omega)},
\end{equation} 
for some constant $C>0$. We need an approximation of the corrector
$w^\eps$ up to the order $\eps^{d}$ and so that we decompose it as
$w^\eps(\bx_0+\eps \bx)=\Psi^\eps(\bx)+r^\eps(\bx)$, where $r^\eps$ is
a remainder of order $\eps^{d}$ in a sense made precise below.
Finding an asymptotic expression for $u^\eps$ then amounts to
calculating $\Psi^\eps(\bx)$ and showing that $r^\eps$ is indeed of
order $\eps^d$. To this aim, we use (\ref{eq:int}) to obtain an
integral equation for $w^\eps$ verified \textit{a.e.} in $\Omega$:
\begin{eqnarray}
w^\eps(\by)&=&-\eps^{d}\int_{B} D_1\left(\bx \right) 
   \nabla \big[w^\eps+U\big](\bx_0+\eps \bx ) \cdot   
  \nabla_\bx N(\bx_0+\eps\bx,\by) d\bx.
\label{eq:w2} 
\end{eqnarray}
We then decompose $N(\bx,\by)$ following (\ref{decompNdiff2}).
Plugging $(\ref{decompNdiff2})$ into $(\ref{eq:w2})$, setting
$\by:=\bx_0+\eps \by$ for $\by \in B$, and using the homogeneity
$\nabla\Gamma(\eps\bx)=\eps^{1-d}\nabla\Gamma(\bx)$, we find
\begin{eqnarray*}
w^\eps(\bx_0+\eps \by)&=&-\eps \int_{B} D_1\left(\bx \right)D_0^{-1}(\bx_0+\eps \bx) \nabla \big[w^\eps+U\big](\bx_0+\eps \bx ) \cdot   \nabla_\bx \Gamma(\bx-\by) d\bx\\
& &-\eps^d \int_{B} D_1\left(\bx \right) \nabla \big[w^\eps+U\big](\bx_0+\eps \bx ) \cdot   \nabla_\bx R_2(\bx_0+\eps \bx,\bx_0+\eps \by) d\bx.
\end{eqnarray*}
We shall prove that the contribution involving $R_2$ above is of order
$\calO(\eps^{d})$ and that up to an error of the same order, we may
replace $D_0^{-1}(\bx_0+\eps \bx)$ and $U(\bx_0+\eps \bx)$ by
$D_{0,d}^{-1}(\bx_0+\eps \bx)$ and $U_d(\bx_0+\eps \bx)$,
respectively, where for $H=D_0^{-1}$ and $H=U$, we have defined the
Taylor expansion to order $d$:
\begin{equation}
  \label{eq:approxd}
  H_{d}(\bx_0+\eps \bx)=\sum_{|m|=0}^d \dfrac{\eps^{|m|}}{m!}\,
   \left(\partial^m H\right) (\bx_0)\, \bx^m.
\end{equation}
Note that $\eps\nabla w^\eps(\bx_0+\eps
\by)=\nabla\Psi^\eps(\by)+\nabla r^\eps(\by)$. We thus want
$\Psi^\eps(\by)$ to solve:
\begin{eqnarray} \label{eq:Psiproof}
\Psi^\eps(\by)+T_{0,d}\Psi^\eps(\by) &=& -\eps \int_{B} D_1\left(\bx \right)D_{0,d}^{-1}(\bx_0+\eps \bx) \nabla U_d(\bx_0+\eps \bx ) \cdot   \nabla_\bx \Gamma(\bx-\by) d\bx,\qquad
\end{eqnarray} 
where we have introduced the notation
\begin{equation}
  \label{eq:T0d}
  T_{0,d} \Psi(\by) = \int_{B} D_1\left(\bx \right) 
   D_{0,d}^{-1}(\bx_0+\eps \bx) \nabla \Psi(\bx ) \cdot   
  \nabla_\bx \Gamma(\bx-\by) d\bx.
\end{equation}
The above equation is the integral formulation of 
\begin{displaymath}
  \begin{array}{rcl}
   \Delta \Psi^\eps +\nabla\cdot\big(D_1(\bx) D_{0,d}^{-1}(\bx_0+\eps\bx)\big)
  \nabla\Psi^\eps = -\eps \nabla \cdot\big( D_1(\bx) D_{0,d}^{-1}(\bx_0+\eps\bx)
   (\nabla U_d)(\bx_0+\eps\bx)\big).
  \end{array}
\end{displaymath}
We now thus expand $D_{0,d}^{-1}(\bx_0+\eps \bx)$ in the definition of
$T_{0,d}$ to obtain:
\begin{eqnarray*}
T_{0,d} \Psi(\by) &=& T_0 \Psi(\by)+\sum_{|m|=1}^{d} \frac{\eps^{|m|}}{m!}\left(\partial^m D_0^{-1}\right) (\bx_0) 
 \int_{B} D_1\left(\bx \right) \bx^m \nabla \Psi(\bx ) \cdot  \nabla_\bx \Gamma(\bx-\by) d\bx,\\
T_{0} \Psi(\by) &:=& \int_{B} D_1\left(\bx \right)
   D_{0}^{-1}(\bx_0) \nabla \Psi(\bx ) \cdot   
  \nabla_\bx \Gamma(\bx-\by) d\bx.
\end{eqnarray*}
Expanding $U_d $ and $D_{0,d}^{-1}$ in (\ref{eq:Psiproof}), and setting 
\begin{eqnarray*} 
\Psi^\eps(\by)&=&D_0(\bx_0)\sum_{|j|=1}^{d}\sum_{|k|=0}^{d}\frac{\eps^{|j|+|k|}}{j!k!} \left(\partial^j U\right) (\bx_0)
\left(\partial^k D_0^{-1}\right) (\bx_0) \, \Psi^\eps_{jk}(\by),
\end{eqnarray*}
leads to the following equation for $\Psi^\eps_{jk}$:
\begin{eqnarray*}
(I+T_0) \Psi^\eps_{jk}(\by) &=&-\sum_{|m|=1}^{d} \frac{\eps^{|m|}}{m!}\left(\partial^m D_0^{-1}\right) (\bx_0) 
 \int_{B} D_1\left(\bx \right) \bx^m \nabla \Psi^\eps_{jk}(\bx ) \cdot  \nabla_\bx \Gamma(\bx-\by) d\bx,\\
&&-D_0(\bx_0)^{-1}\int_{B} D_1\left(\bx \right) \bx^k \nabla \bx^j\cdot   \nabla_\bx \Gamma(\bx-\by) d\bx.
\end{eqnarray*}
Equating like powers of $\eps$, we verify that 
\begin{math}
  \Psi^\eps_{jk}(\by)=\sum_{l=0}^d
\frac{\eps^{l}}{l!} \phi_{jk}^l(\by),
\end{math}
where $\phi_{jk}^l$ solves the following integral equation
$\textit{a.e.}$ in every bounded set of $\Rm^d$:
\begin{eqnarray*}
(I+T_0) \phi_{jk}^l(\by)&=& 
-\sum_{|m|=1}^{l} \frac{l!(\partial^m D_0^{-1})(\bx_0)}{m! (l-|m|)!}   \int_{B} D_1\left(\bx \right) \bx^m \nabla \phi_{jk}^{l-|m|}(\bx) \cdot   \nabla_\bx \Gamma(\bx-\by) d\bx\\
&&-\delta_l^0D_0^{-1}(\bx_0)\int_{B} D_1\left(\bx \right) \bx^k \nabla \bx^j \cdot   \nabla_\bx \Gamma(\bx-\by) d\bx.
\end{eqnarray*}
Existence and uniqueness of solutions in $H^1_\textrm{loc}(\Rm^d)\cap
\calC^\infty(\Rm^d \backslash \overline{B})$ to the above equations
follows from lemma \ref{lemappend1} of the appendix: we first prove
the result for $\phi_{jk}^0$, then for $\phi_{jk}^1$ which depends
only on $\phi_{jk}^0$, and finally for all $\phi_{jk}^m$ iteratively.
Moreover, according to the lemma, $\phi_{jk}^l$ solves the system of
differential equations given in \eqref{eq:phijkl1}.  The function
$\Psi^\eps$ thus belongs to the space $H^1_\textrm{loc}(\Rm^d)\cap
\calC^\infty(\Rm^d \backslash \overline{B})$ by construction.  When
$D_0$ is constant and equal to $D_0(\bx_0)$ in the set $\bx_0+\eps B$,
we do not need to expand $D_0^{-1}$. We thus have
$D_{0,d}^{-1}=D_0^{-1}(\bx_0)$ and $ \phi_{j0}^0$ can be identified
with $\Psi^\eps_{j0}$.

We then verify that the remainder
$r^\eps(\by)=w^\eps(\bx_0+\eps\by)-\Psi^\eps(\by)$ belongs to $H^1(B)$
by construction and moreover solves the integral equation:
\begin{displaymath}
  \begin{array}{rcl}
(I+T_{0,d})r^\eps(\by)= S^\eps(\eps\by)
   - \eps^{d+2} \!\dint_B\!
   D_1(\bx) \big[S_1(\bx)\nabla w^\eps(\bx_0+\eps\bx)
   + \bS_2(\bx)\big]\!\cdot\!\nabla_\bx\Gamma(\bx-\by) d\bx 
  \end{array}
\end{displaymath}
where $S_1$ is the remainder of the $d+1$ order Taylor expansion of
$D_0^{-1}(\bx_0+\eps \bx)$ (so that $D_0^{-1}(\bx_0+\eps
\bx)=D_{0,d}^{-1}(\bx_0+\eps \bx)+\eps^{d+1}S_1(\bx)$),  $\bS_2$
the remainder of $D_0^{-1}(\bx_0+\eps \bx) \nabla U(\bx_0+\eps \bx )$
and where we have defined
\begin{displaymath}
  S^\eps(\eps \by) = 
  -\eps^d\int_{B} D_1\left(\bx \right) \nabla u^\eps(\bx_0+\eps \bx ) 
   \cdot   \nabla_\bx R_2(\bx_0+\eps \bx,\bx_0+\eps \by) d\bx.
\end{displaymath}
We may now decompose $r^\eps$ as
$r^\eps(\by):=r^\eps_1(\by)+r^\eps_2(\by)+S^\eps(\eps \by)$ with:
\begin{eqnarray*}
(I+T_{0,d}) r^\eps_1(\by)&=& -
 \eps
  \int_{B} D_1(\bx) D_{0,d}^{-1}(\bx_0+\eps \bx)  \nabla S^\eps(\eps \by)\cdot \nabla_\bx \Gamma(\bx-\by) d\bx,\\
(I+T_{0,d}) r^\eps_2(\by)&=& -\eps^{d+2}  \int_{B} D_1\left(\bx \right) \big[S_1(\bx)\nabla w^\eps(\bx_0+\eps\bx)
   + \bS_2(\bx)\big] \cdot \nabla_\bx \Gamma(\bx-\by) d\bx. \nonumber 
\end{eqnarray*} 
We know from the hypotheses in (\ref{eq:hypD}) that for all $\by \in
\overline{B}$, $D_0(\bx_0+\eps\by)+D_1(\bx)\geq C_0>0$, so that
setting $0<\eps \leq \eps_0 $ for $\eps_0$ small enough, we have
$1+D_1(\bx)D^{-1}_{0,d}(\bx_0+\eps\by)\geq C_1>0$, for another
constant $C_1$ independent of $\eps$. An application of lemma
\ref{lemappend1} then yields that $r_1^\eps$ and $r_2^\eps$ are
uniquely defined in $H^1_\textrm{loc}(\Rm^d)\cap \calC^\infty(\Rm^d
\backslash \overline{B})$.  Moreover, following lemma
\ref{lemappend1}, we have the estimates:
\begin{eqnarray*} 
\|\nabla r_1^\eps \|_{L^2(\Rm^d)} &\leq& C 
 \eps \|D_1\|_{L^\infty(B)}\| \nabla S^\eps (\eps\cdot)\|_{L^\infty(B)},\\
\|\nabla r_2^\eps \|_{L^2(\Rm^d)} &\leq&  C \eps^{d+2} \|D_1\|_{L^\infty(B)} \left(\| \nabla w^\eps(\bx_0+\eps \cdot) \|_{L^2(B)}+\| D_0^{-1} \nabla U\|_{\calC^{d+1}(B_0)} \right) ,\nonumber\\
&\leq &C \eps^{d+2} \|D_1\|^2_{L^\infty(B)} \|g\|_{L^2(\partial \Omega)}, \label{estimr2}
\end{eqnarray*} 
according to (\ref{estim:w}) and by elliptic regularity, where $B_0$
is as above \eqref{estim:w}. It thus remains to estimate $S^\eps$.
$\mbox{}$From proposition \ref{prop:decompNdiff}, we know that $R_2\in
\calC^\infty(\Omega \times \Omega)$, which yields:
\begin{eqnarray*}
\| \nabla S^\eps(\eps\cdot) \|_{L^\infty(B)} &\leq& C \eps^d \|D_1\|_{L^\infty(B)} \left(\| \nabla w^\eps(\bx_0+\eps \cdot) \|_{L^2(B)}+ \| \nabla U \|_{L^\infty(B_0)}\right) \times \\
  &&\times \| \nabla_\bx \nabla_\by R_2 \|_{L^\infty(B_0\times B_0)}
\qquad\leq\qquad \eps^d \|D_1\|^2_{L^\infty(B)} \|g\|_{L^2(\partial \Omega)}.
\end{eqnarray*}
Gathering the different estimates for $r^\eps_1$, $r^\eps_2$ and
$S^\eps$, we obtain that
$$
\|\nabla r^\eps\|_{L^2(B)} \leq  C \eps^{d+1} \|D_1\|^2_{L^\infty(B)} \|g\|_{L^2(\partial \Omega)}.
$$
To conclude the proof, we go back to (\ref{eq:int}) and take the
trace on $\partial \Omega$. Plugging $\nabla w^\eps(\bx_0+\eps
\bx)=\eps^{-1}(\nabla \Psi^\eps(\bx)+\nabla r^\eps(\bx))$ into
(\ref{eq:int}), it just remains to expand $\nabla U(\bx_0+\eps \bx)
\in \calC^\infty(\overline{B})$ and $ \nabla_\bx N(\bx_0+\eps
\bx,\by)$ thanks to (\ref{decompNdiff3}) since we find, \textit{a.e.}
in $ \partial \Omega$, that:
\begin{eqnarray*}
\left. u^\eps(\by) \right|_{\partial \Omega}\!\!&=&\!\!
  \left.  U(\by)  \right|_{\partial \Omega}-\eps^d\! \int_{B}  \!\!
D_1(\bx)\Big(\nabla U(\bx_0+\eps \bx)+ \dfrac 1\eps \nabla  \Psi^\eps(\bx) \Big)\!\cdot \!
\nabla_\bx \left. N(\bx_0+\eps \bx,\by) \right|_{\partial \Omega} d\bx\\
&&\!+\calO(\eps^{2d}).
\end{eqnarray*}
The asymptotic expansion of remark \ref{rem:alt} is obtained by
decomposing $w^\eps$ slightly differently. We write
$w^\eps(\bx_0+\eps\bx)=\Psi^\eps(\bx)+r^\eps(\bx)$, where $\Psi^\eps$
is now  given by
\begin{eqnarray*}
\Psi^\eps(\by)+T^\eps\Psi^\eps(\by) &=& -\eps \int_{B} D_1\left(\bx \right)D_{0}^{-1}(\bx_0+\eps \bx) \nabla U_d(\bx_0+\eps \bx ) \cdot   \nabla_\bx \Gamma(\bx-\by) d\bx,\qquad\\
T^\eps\Psi(\by) &:=& \int_{B} D_1\left(\bx \right)
   D_{0}^{-1}(\bx_0+\eps \bx) \nabla \Psi(\bx ) \cdot   
  \nabla_\bx \Gamma(\bx-\by) d\bx.
\end{eqnarray*} 
We then verify that the remainder $r^\eps$ is of order $\eps^d$ and that expanding $U_d$ and setting
\begin{math}
  \Psi^\eps(\bx)=\sum_{|j|=1}^d \frac{\eps^{|j|}}{j!}\,
   \left(\partial^j U\right) (\bx_0)\, \Psi^\eps_j,
\end{math}
leads to the desired result.
\end{proofof}
\bigskip

\begin{proofof}{\em Theorem \ref{th2}}. 
  Let $D_1$ be a non-regular perturbation and let $\chi^\eta$ be the
  cut-off function with support in $B$ defined as
 \begin{equation}
\label{eq:cut}
 \left\{ 
\begin{array}{l}
 \chi^\eta(\bx)=1, \quad \textrm{for } \bx \in B   \textrm{ such that dist}(\bx,\partial B) > \eta,\\
 \chi^\eta(\bx)=0, \quad \textrm{otherwise}.
\end{array} \right.
\end{equation}
The parameter $\eta$ will be adjusted according to $\eps$. Let now
$\rho^\eta(\bx):=\eta^{-d}\rho(\eta^{-1} \bx)$ be a standard mollifier
and let $D_1^\eta:=\rho^\eta*(\chi^\eta D_1)$. We verify that
$D_1^\eta \in \calC^\infty(\Rm^d)$ and that its support is included in
$B$ with a vanishing and continuous trace at the boundary. We can then
apply theorem \ref{th:asymp} to obtain an asymptotic expansion for the
solution $u^\eps_\eta$ associated to $D_1^\eta$. Since the error term
of order $\eps^{2d}$ depends only on $\| D_1^\eta
\|_{L^\infty(\Rm^d)}$ - which is bounded by $\|
D_1\|_{L^\infty(\Rm^d)}$ - it suffices to look at the limit of the
different polarization tensors to find the limiting asymptotic
expansion.

Since $D_0(\bx)+D_1\left(\frac{\bx - \bx_0}{\eps}\right)$ is bounded
from below by $C_0$, this property is still verified by the
regularized diffusion coefficient so that, according to
(\ref{estimappendphi}) of lemma \ref{lemappend1} of the appendix, the
function $\phi_{jk}^{l,\eta}$ associated to $D_1^\eta$ satisfy by
induction the estimates, for $l=0,\cdots,d$:
$$
\| \nabla \phi_{jk}^{l,\eta} \|_{L^2(\Rm^d)} \leq C \| D_1^\eta\|^{l+1}_{L^\infty(\Rm^d)}, \qquad \| \phi_{jk}^{l,\eta} \|_{L^2(A)} \leq C \| D_1^\eta\|^{l+2}_{L^\infty(\Rm^d)},
$$
for any bounded set $A$. This yields that $\nabla
\phi_{jk}^{l,\eta}$ is bounded in $L^2(\Rm^d)$ independently of $\eta$
and so is $\phi_{jk}^{l,\eta}$ in $H^1(A)$.  Defining the set
$E:=\{(j,k)\in \NN^{2d}, l \in \NN, \, 0\leq |j|, |k|, l \leq d \}$
with cardinal $|E|$, we may thus see $\{\phi_{jk}^{l,\eta}\}_{E}$ as
bounded in $(H^1(A))^{|E|}$ and extract a subsequence as $\eta\to0$
converging strongly in $(L^2(A))^{|E|}$ and with gradient converging
weakly in $(L^2(\Rm^d))^{|E|}$ to a limit $\{\phi_{jk}^{l}\}_{E}$. We
obtain that $\phi_{jk}^{l}\in H^1_\textrm{loc}(\Rm^d)$ and $\nabla
\phi_{jk}^{l} \in L^2(\Rm^d)$.  To find the equation solved by $\nabla
\phi_{jk}^{l}$, we consider the weak formulation verified by
$\phi_{jk}^{l,\eta}$, which is, for all functions $\varphi \in
H^1_\textrm{loc}(\Rm^d)$ such that $\nabla \varphi \in L^2(\Rm^d)$,
$R^{-d} \|\varphi \|_{L^1(S_R)} \to 0$ as $R \to \infty$, where $S_R$
denotes the sphere of radius $R$,
$$
\begin{array}{l}
\DST \int_{\Rm^d }\left(D_0(\bx_0)+D^\eta_1(\bx) \right)\nabla \phi_{jk}^{l,\eta} \cdot \nabla \varphi\, d\bx=
\DST -\,\delta_l^0\, \int_{B } D^\eta_1\left(\bx \right) \bx^k \nabla \bx^j \cdot \nabla \varphi \, d\bx \\
\DST \quad -D_0(\bx_0)\sum_{|m|=1}^{l} \frac{l!}{m! (l-|m|)!}\left(\partial^m D_0^{-1}\right) (\bx_0)\int_{B }D^\eta_1\left(\bx \right) \bx^m \nabla \phi_{jk}^{l-|m|,\eta}(\bx) \cdot \nabla \varphi \, d\bx.
\end{array}
$$
The above formulation is justified in lemma \ref{lemappend1} below; see
(\ref{weakform}). Since $D_1^\eta$ converges strongly in any
$L^p(\Rm^d)$ for $1 \leq p < \infty$, we can  pass to the limit
in the non-linear terms above and obtain the following limiting equation:
\begin{equation}
\begin{array}{l}
\DST \int_{\Rm^d }\left(D_0(\bx_0)+D_1(\bx) \right)\nabla \phi_{jk}^{l} \cdot \nabla \varphi\, d\bx=-\,\delta_l^0\, \int_{B } D_1\left(\bx \right) \bx^k \nabla \bx^j \cdot \nabla \varphi \, d\bx\\
\DST \hspace{0.1cm} -D_0(\bx_0)\sum_{|m|=1}^{l} \frac{l!}{m! (l-|m|)!}\left(\partial^m D_0^{-1}\right) (\bx_0)\int_{B }D_1\left(\bx \right) \bx^m \nabla \phi_{jk}^{l-|m|}(\bx) \cdot \nabla \varphi \, d\bx.
\end{array} \label{limphi}
\end{equation}
To obtain the behavior of $\phi_{jk}^{l}$ at infinity, we use the
integral formulation given in (\ref{integappend}) of lemma
\ref{lemappend1} for the subsequence $\phi_{jk}^{l,\eta}$ and obtain,
\textit{a.e.} in every bounded set $\Omega' \subset \Rm^d$:
\begin{eqnarray*} 
(I+T_0)\phi_{jk}^{l,\eta}(\by)&=&
-\sum_{|m|=1}^{l}  \frac{l!\partial^m D_0^{-1} (\bx_0)}{m! (l-|m|)!} \int_{B} D_1^\eta\left(\bx \right) \bx^m \nabla \phi_{jk}^{l-|m|,\eta}(\bx) \cdot   \nabla_\bx \Gamma(\bx-\by) d\bx\\
&&-\delta_l^0 D_0^{-1}(\bx_0)\int_{B} D_1^\eta\left(\bx \right) \bx^k \nabla \bx^j \cdot   \nabla_\bx \Gamma(\bx-\by) d\bx.
\end{eqnarray*}
The above equation makes sense in $L^2(\Omega')$ and therefore almost
everywhere in $\Omega'$ since $\nabla \phi_{jk}^{l,\eta} \in
L^2(\Omega)$, $l=0,\cdots,d$, and $\nabla \Gamma \in
L^1_\textrm{loc}(\Rm^d)$ so that the right hand side is finite thanks
to the Young inequality. Consider now a compact set $K \subset \Rm^d$
such that $\textrm{dist}(K, B)>C>0$. The above equation is then
verified uniformly in $K$ and moreover $\phi_{jk}^{l,\eta} \in
\calC^0(K)$. Since $\nabla \phi_{jk}^{l,\eta}$ converges weakly to
$\nabla \phi_{jk}^{l}$ for $0\leq l\leq d$ and $D_1^\eta$ converges
strongly, it follows from the above equation that $\phi_{jk}^{l,\eta}$
is a Cauchy sequence in $\calC^0(K)$ so that it converges uniformly to
the solution, for all $\bx\in K$, to
\begin{eqnarray}
(I+T_0)\phi_{jk}^{l}(\by)&=&
- \sum_{|m|=1}^{l} \frac{l!\partial^m D_0^{-1} (\bx_0)}{m! (l-|m|)!} \int_{B} D_1\left(\bx \right) \bx^m \nabla \phi_{jk}^{l-|m|}(\bx) \cdot   \nabla_\bx \Gamma(\bx-\by) d\bx \nonumber \\
&&-\delta_l^0 D_0^{-1}(\bx_0) \int_{B} D_1\left(\bx \right) \bx^k \nabla \bx^j \cdot   \nabla_\bx \Gamma(\bx-\by) d\bx. 
\label{intphilim} 
\end{eqnarray}
The fact that $\nabla_\bx \Gamma(\bx-\by)=\calO(|\by|^{1-d})$ for $\bx
\in B$ and $\by \in K$ yields that
$\phi_{jk}^{l}(\by)=\calO(|\by|^{1-d})$ for such values of $\by$. It
is then not difficult to see that (\ref{limphi}) is the weak
formulation of the problem given in the theorem. Notice that equation
(\ref{intphilim}) is also valid \textit{a.e.} in $A$ since
$\phi_{jk}^{l}\in L^2(A)$ for any bounded set $A$. Uniqueness follows
from (\ref{limphi}) and the behavior at infinity: the right hand side
of (\ref{limphi}) vanishes when we consider the difference of two
possible solutions. Since those solutions are sufficiently regular,
taking that difference as a test function implies the difference is a
constant which must be equal to zero according to the vanishing limit
at infinity.

Now that we have the expression of the limiting $\phi_{jk}^{l}$, it
suffices to pass to the limit in the polarization tensors using the
weak convergence of $\nabla \phi_{j\, 0}^{0,\eta}$ and the strong
convergence of $D_1^\eta$ and to choose $\eta$ small enough such that
all the errors terms coming from the different passages to the limit
are smaller than $C \eps^{2d}$.
\end{proofof}
\bigskip

\begin{proofof}{\em Proposition \ref{prop:jump}}.
  When $D_0$ is constant on the set $\bx_0+\eps B$, only the sum
  involving the polarization tensor $M$ remains in theorem
  \ref{th:asymp} as we have mentioned in remark \ref{rem2}. We thus
  start from the expression of $M$ given in theorem \ref{th2} and
  define $f_j:=\phi^0_{j0}-\phi_j$. A proof of the existence and
  uniqueness for $\phi_j$ can be found in \cite{AK-04}. Owing the
  definitions of $\phi^0_{j0}$ and $\phi_j$, we find that $f_j \in
  H^1_{\textrm{loc}}(\Rm^d)\cap \calC^\infty(\Rm^d/\overline{B})$ by
  construction and is the unique weak solution to
\begin{eqnarray*}
\nabla \cdot (D_0+\un_B(\bx) D_1) \nabla f_j &=-\un_B(\bx) D_1\Delta \bx^j, \quad \bx \in \Rm^d, 
\end{eqnarray*}
equipped with the condition at infinity:
$$
f_j(\by)+\Gamma(\by) D_0^{-1}D_1 \int_{\partial B} \bn \cdot \nabla
\bx^j d\sigma(\bx)=\calO(|\by|^{1-d}).
$$
Here, $\un_B$ is the characteristic function of the set $B$. When
$|j|=1$, we obtain $f_j=0$. When $|j| \geq 2$, we need to sum over $j$
to show that $f_j$ is small in an appropriate sense. To this aim, we
derive an integral equation for $f_j$ from that of $\phi^0_{j0}$ and
$\phi_j$. As we mentioned in the proof of theorem \ref{th2},
(\ref{intphilim}) is verified \textit{a.e.} by $\phi^0_{j0}$ so that
we have
\begin{equation} \label{integphi00} D_0 D_1^{-1} \phi_{j0}^{0}(\by)= -\int_{B} \nabla \phi_{j0}^{0}(\bx) \cdot   \nabla_\bx \Gamma(\bx-\by) d\bx -\int_{B} \nabla \bx^j \cdot   \nabla_\bx \Gamma(\bx-\by) d\bx.
\end{equation}
Since $\phi_j$ is harmonic in $B \cup \Rm^d \backslash \overline{B}$,
we deduce from elliptic regularity in Lipschitz domains (see e.g.
\cite{AK-04}) that $\phi_j \in H^{\frac{3}{2}}(B)$ so that its inner
normal derivative at the boundary $\partial B$ belongs to
$L^2(\partial B)$. This allows us to express $\phi_j$ in terms of
single layer potential, using the jump of its normal derivative at the
boundary given in proposition \ref{prop:jump}, as
\begin{equation} \label{single}
D_0 D_1^{-1}\phi_j(\by)=- \int_{\partial B}\left(\left.\frac{\partial  \phi_{j}}{\partial \bn }\right|_-(\bx) + \bn \cdot \nabla \bx^j \right) \Gamma(\bx-\by) d\sigma(\bx).
\end{equation}
The latter equation is verified in $L^1(A)$ for any bounded set $A
\subset \Rm^d$, and thus $\textit{a.e.}$ since
$\|\Gamma(\bx-\cdot)\|_{L^1(A)}$ is uniformly bounded in $\bx \in
\partial B$. Moreover, since $\phi_j$ is harmonic in $B$, we have for
any $\varphi \in H^1(B)$:
$$
\int_B \nabla \phi_j  \cdot \nabla \varphi \,  d\bx =\int_{\partial B}\left.\frac{\partial  \phi_{j}}{\partial \bn }\right|_- \varphi \,  d\sigma(\bx).
$$
Let $\psi \in \calC^0_c(A)$ and set $\varphi(\bx)=\int_{A}
\psi(\by) \Gamma(\bx-\by) d\by$. Using the Young inequality and the
fact that $\Gamma$ and $\nabla \Gamma$ belong to
$L^1_\textrm{loc}(\Rm^d)$, we verify that $\varphi \in H^1(B)$ so that
it can be used as a test function. Moreover, to be able to use the
Fubini theorem, we apply as in the proof of lemma \ref{lemappend1} the
Sobolev inequality \ref{lem:sobolev} to conclude that $\nabla
\phi_j(\bx) \cdot \nabla \Gamma(\bx-\by) \psi(\by)$ belongs to
$L^1(B\times A)$. In the same way, $\left.\frac{\partial
    \phi_{j}}{\partial \bn }\right|_-(\bx) \Gamma(\bx-\by)\psi(\by)$
belongs to $L^1(\partial B\times A)$ since
$$
\int_{\partial B} \int_{A} \left|\left.\frac{\partial  \phi_{j}}{\partial \bn }\right|_-(\bx) \Gamma(\bx-\by)\psi(\by) \right| d\sigma(\bx)d\by \leq C \left\| \left.\frac{\partial  \phi_{j}}{\partial \bn }\right|_{-}  \right\|_{L^2(\partial B)} \|\Gamma\|_{B_a} \|\psi\|_{L^\infty(A)},
$$
for a ball of radius $a$ large enough. We may thus write:
$$
\int_A \left(\int_B \nabla \phi_j(\bx)  \cdot \nabla \Gamma(\bx-\by) d\bx-\int_{\partial B}\left.\frac{\partial  \phi_{j}}{\partial \bn }\right|_- \Gamma(\bx-\by) \,  d\sigma(\bx) \right) \psi(\by)d\by=0.
$$
Plugging (\ref{single}) into the latter equation yields:
$$
\int_A \left(D_0 D_1^{-1}\phi_j(\by)+\int_B \nabla \phi_j(\bx)  \cdot \nabla \Gamma(\bx-\by) d\bx+\int_{\partial B} \bn \cdot \nabla \bx^j\, \Gamma(\bx-\by) \,  d\sigma(\bx) \right) \psi(\by)d\by=0.
$$
Integrating (\ref{integphi00}) against $\psi$, subtracting the
equation above and performing an integration by parts, we find:
$$
\int_A \left(D_0 D_1^{-1}f_j(\by)+\int_B \nabla f_j(\bx)  \cdot \nabla \Gamma(\bx-\by) d\bx-\int_{B} \Delta \bx^j\, \Gamma(\bx-\by) \,  d\bx \right) \psi(\by)d\by=0.
$$
The quantity under parentheses belongs to $L^2(A)$. Thus, we
deduce by density that the above relation holds also for any $\psi \in
L^2(A)$ so that $f_j$ solves the following integral equation,
\textit{a.e.} in every bounded set $\Omega' \subset \Rm^d$,
\begin{eqnarray} \label{integf}
D_0 D_1^{-1} f_{j}(\by)&=& \int_{B} \nabla f_{j}(\bx) \cdot   \nabla_\bx \Gamma(\bx-\by) d\bx+\int_{B}\Delta \bx^j \Gamma(\bx-\by) d\sigma(\bx).
\end{eqnarray}
We now show that an appropriate linear combination of the $f_j$'s is of order
$\eps^{d+1}$. First, since $D_0$ is constant in $\bx_0+\eps B$,
$\Delta U(\bx_0+\eps \bx)=0$ for $\bx \in B$ according to
(\ref{eq:diffD0}), so that, using the notation in \eqref{eq:approxd},
we get that $\Delta U_d(\bx)=\Delta
(U_d(\bx)-U(\bx_0+\eps\bx))=\calO(\eps^{d+1})$ uniformly in $B$. As a
consequence, we have
 $$
R^\eps(\bx):= \Delta U_d (\bx) = \sum_{|j|=1}^d
\frac{\eps^{|j|}}{j!}\partial^j U(\bx_0) \Delta \bx^j
=\calO(\eps^{d+1}).
$$ 
Thus, defining 
$$F^\eps(\bx):=\sum_{|j|=1}^d \frac{\eps^{|j|}}{j!}\partial^j U(\bx_0)  f_j(\bx),$$
it follows:
\begin{eqnarray*}
\sum_{|j|=1}^d \frac{\eps^{|j|}}{j!}\partial^j U(\bx_0)  M_{ij}=D_1 \sum_{|j|=1}^d \frac{\eps^{|j|}}{j!}\partial^j U(\bx_0) \int_{B} \nabla (\bx^j+\phi_{j}(\bx)+f_j(\bx)) \cdot \nabla \bx^i d \bx,\\
=D_1 \sum_{|j|=1}^d \frac{\eps^{|j|}}{j!}\partial^j U(\bx_0) \int_{B} \nabla (\bx^j+\phi_{j}(\bx)) \cdot \nabla \bx^i d\bx
+D_1 \int_{B} \nabla F^\eps(\bx) \cdot \nabla \bx^i d\bx.
\end{eqnarray*}
According to the definition of $f_j$, $F^\eps \in
H^1_\textrm{loc}(\Rm^d)\cap \calC^\infty(\Rm^d/\overline{B})$ solves:
\begin{eqnarray}
\nabla \cdot (D_0(\bx_0)+\un_B(\bx) D_1) \nabla F^\eps&=-\un_B(\bx) D_1 R^\eps, \quad \bx \in \Rm^d, \label{Fprop}
\end{eqnarray}
equipped with the condition at infinity:
\begin{equation} 
\label{condinf}
F^\eps(\by)+\Gamma(\by) D_1\int_{B} R^\eps d\bx=\calO(|\by|^{1-d}).
\end{equation}
Following (\ref{integf}), $F^\eps$ thus solves the integral equation,
\textit{a.e.} in every bounded set $\Omega' \subset \Rm^d$:
\begin{eqnarray}
 \label{eq:Fepsint} 
F^\eps(\by)&=& -D_1D_0^{-1}(\bx_0) \int_{B} \big[\nabla F^\eps(\bx) \cdot   \nabla_\bx \Gamma(\bx-\by) -R^\eps(\bx)\Gamma(\bx-\by)\big] d\bx,
\end{eqnarray}
so that Young's inequality gives
\begin{equation} \label{L2F}
\| F^\eps \|_{L^2(B)} \leq C \| \nabla F^\eps \|_{L^2(B)} + C\| R^\eps \|_{L^2(B)} \leq C \| \nabla F^\eps \|_{L^2(B)}+\calO(\eps^{d+1}).
\end{equation}
Let $B_R$ be the ball of radius $R$ with $B \subset \subset B_R$ and
denote by $S_R$ its boundary. Integrating (\ref{Fprop}) on $B_R$
against $F^\eps$ leads to
\begin{equation} \label{weakFproof}
\int_{B_R} (D_0(\bx_0)+\un_B D_1)|\nabla F^\eps|^2 d\bx=-D_1 \int_B R^\eps F^\eps d\bx+\int_{S_R} \frac{\partial F^\eps}{\partial \bn} F^\eps d\sigma(\bx), 
\end{equation}
where $\sigma$ is the surface measure on $S_R$.  We may recast
condition (\ref{condinf}) using the integral equation
\eqref{eq:Fepsint} for $F^\eps$ and its derivative as
\begin{equation}
  \label{eq:Fepsalpha}
  \partial^\alpha F^\eps(\by)+\partial^\alpha\Gamma(\by) D_1\int_{B} R^\eps d\bx=\calO\left(( \|\nabla F^\eps\|_{L^2(B)}+ \|R^\eps\|_{L^2(B)}) |\by|^{1-d-|\alpha|}\right),
\end{equation}
for a multi-index $\alpha$ with $|\alpha| \leq 1$.  Consider first
$d\geq 3$. Then $\nabla F^\eps \in L^2(\Rm^d)$ and the boundary
integral in \eqref{weakFproof} goes to zero as $R$ tends to infinity
so that
\begin{equation} \label{weakF}
\int_{\Rm^d} (D_0(\bx_0)+\un_B D_1)|\nabla F^\eps|^2 d\bx=-D_1 \int_B R^\eps F^\eps d\bx. 
\end{equation}
This yields, together with (\ref{L2F}):
$$
\| \nabla F^\eps\|^2_{L^2(\Rm^d)} \leq  \calO(\eps^{2(d+1)})+\| R^\eps \|_{L^2(B)}\| \nabla F^\eps\|_{L^2(\Rm^d)},
$$
so that $\| \nabla F^\eps\|_{L^2(\Rm^d)}=\calO(\eps^{d+1})$.
Consider now the case $d=2$. We cannot use the same approach since
$F^\eps$ does not vanish at infinity.  Using \eqref{eq:Fepsalpha} for
$\by \in S_R$, we have
\begin{eqnarray*}
|F^\eps(\by)|& \leq &C \Big(\Big(\log R+\frac{1}{R}\Big)\|R^\eps\|_{L^2(B)}+\frac{1}{R}\|\nabla F^\eps\|_{L^2(B)} \Big),\\
 \Big|\frac{\partial F^\eps}{\partial \bn}(\by)\Big|& \leq &C \Big(\frac{1}{R}\Big(1+\frac{1}{R}\Big)\|R^\eps\|_{L^2(B)}+\frac{1}{R^2}\|\nabla F^\eps\|_{L^2(B)} \Big),
\end{eqnarray*}
so that, since $\|  R^\eps\|_{L^2(B)}=\calO(\eps^{d+1})$,
\begin{eqnarray*}
\Big|\int_{S_R} \frac{\partial F^\eps}{\partial \bn} F^\eps d\sigma(\bx) \Big|&\leq&
C \eps^{2(d+1)}+C_R\, \eps^{d+1} \|\nabla F^\eps\|_{L^2(B)}+\frac{C}{R^3}\|\nabla F^\eps\|^2_{L^2(B)}.
\end{eqnarray*}
Since, according to hypothesis \ref{eq:hypD},
$D_0(\bx_0)+\un_B D_1 \geq C_0>0$ \textit{a.e.} in $\Rm^d$, it follows from (\ref{L2F}), (\ref{weakFproof}) and the above inequality that:
$$
C_0 \|\nabla F^\eps\|_{B_R}^2 \leq C \eps^{2(d+1)}+\Big(\frac{C}{R^3}+\eta \Big)\|\nabla F^\eps\|^2_{L^2(B)},
$$
for any $\eta>0$. It suffices finally to set $\eta$ small enough
and $R$ large enough so that $\frac{C}{R^3}+\eta<C_0$ to obtain
$$
\| \nabla F^\eps\|_{L^2(B)} \leq \| \nabla F^\eps\|_{L^2(B_R)} = \calO(\eps^{d+1}).
$$
We end the proof with the following integration by parts:
\begin{displaymath}
  \begin{array}{rcl}
\DST\sum_{|j|=1}^d \frac{\eps^{|j|}}{j!}\partial^j U(\bx_0)  M_{ij}
&=&\DST D_1 \sum_{|j|=1}^d \frac{\eps^{|j|}}{j!}\partial^j U(\bx_0) \int_{\partial B} \bn \cdot \nabla (\bx^j+\phi_{j}(\bx)) \bx^i d \sigma(\bx)\\
&&\DST-D_1 \int_{B} R^\eps(\bx)\, \bx^i d\bx+D_1 \int_{B} \nabla F^\eps(\bx) \cdot \nabla \bx^i d\bx,\\
&=&\DST D_1 \sum_{|j|=1}^d \frac{\eps^{|j|}}{j!}\partial^j U(\bx_0) \int_{\partial B} \bn \cdot \nabla (\bx^j+\phi_{j}(\bx)) \bx^i d \sigma(\bx)+\calO(\eps^{d+1}),\\
&=&\DST\sum_{|j|=1}^d \frac{\eps^{|j|}}{j!}\partial^j U(\bx_0)  \calM_{ij}+\calO(\eps^{d+1}),
  \end{array}
\end{displaymath}
which shows that the error terms generated by $M$ and $\mathcal M$
agree up to an order $\calO(\eps^{d+1})$.
\end{proofof}

\subsection{Asymptotic expansions for the Helmholtz equation} 
\label{proofprop}

We now prove proposition \ref{prop:asympV}, theorem \ref{th:asympV},
and proposition \ref{prop:equiv}.

\begin{proofof} {\em proposition \ref{prop:asympV}.}
  We write $v^\eps:=V+w^\eps$ so that the corrector $w^\eps$
  satisfies:
\begin{equation}
\label{eq:wproof}
\begin{array}{l}
 \DST-\Delta w^\eps(\bx)+\left(q_0(\bx)+\frac{1}{\eps^{2-\eta}} q_1\left(\frac{\bx-\bx_0}{\eps} \right) \right)\, w^\eps(\bx)=-\frac{1}{\eps^{2-\eta}} q_1\left(\frac{\bx-\bx_0}{\eps} \right) V(\bx), \quad  \bx \in \Omega,\\[2mm]
 \DST  \frac{\partial w^\eps}{\partial \bn } =0,\quad \textrm{on } \partial \Omega.
\end{array} 
\end{equation} 
We need to show the existence of $w^\eps$. We first show the existence
and uniqueness of a solution to the integral formulation of
$(\ref{eq:wproof})$, which formally reads, $\by$ \textit{a.e.} in
$\Omega$:
\begin{eqnarray*} 
w^\eps+T^\eps w^\eps&=& -T^\eps V,\\
T^\eps \varphi(\by) &=& \int_{\bx_0+\eps B} q_1\left(\frac{\bx-\bx_0}{\eps} \right) \varphi(\bx) N(\bx,\by) d\bx.
\end{eqnarray*}
We consider first the case $d \geq 3$. Using the decomposition of $N$
given in proposition \ref{prop:decompN} and denoting by $w^*$ the
restriction of $w^\eps(\bx_0+\eps \by)$ to $B$ (we do not write the
dependence on $\eps$ to simplify), we recast the above system as
\begin{equation} \label{eq:wstar}
 \left\{ 
\begin{array}{l}
\DST w^*+\eps^\eta T w^*+\eps^{d-2+\eta}R^\eps w^*=-T^\eps V(\bx_0+\eps \by), \qquad \by \in B,\\
\DST T w^*(\by)=\int_{B} q_1\left(\bx \right) w^*(\bx) \Gamma(\bx- \by) d\bx,\\
\DST R^\eps w^*(\by)=\int_{B} q_1\left(\bx \right) w^*(\bx) R(\bx_0+\eps \bx, \bx_0+\eps \by) d\bx.
\end{array} \right.
\end{equation}
We have used the homogeneity $\Gamma(\eps\bx)=\eps^{2-d}\Gamma(\bx)$
when $d \geq 3$. Since $T$ and $R^\eps$ are compact operators in
$L^2(B)$, they have discrete spectra. Indeed, since $\Gamma, \nabla
\Gamma \in L^1_\textrm{loc}(\Rm^d)$, we have, using the Young
inequality for any $\varphi \in L^2(B)$,
$$
\|T \varphi \|_{L^2(B)} \leq \| q_1 \varphi \|_{L^2(B)} \| \Gamma \|_{L^1(B_a)} \leq  \| q_1\|_{L^\infty(B)} \| \varphi \|_{L^2(B)} \| \Gamma \|_{L^1(B_a)},
$$
where $B_a$ is a ball of radius $a$ large enough. Thus, proceeding
analogously for $\nabla T$,
$$
\|T \|_{\calL(L^2(B))} \leq \| q_1\|_{L^\infty(B)} \| \Gamma \|_{L^1(B_a)}, \qquad \|\nabla T \|_{\calL(L^2(B))} \leq \| q_1\|_{L^\infty(B)} \| \nabla \Gamma \|_{L^1(B_a)},
$$
and compactness stems from the Rellich theorem.  The same holds for
$R^\eps$ since it is Hilbert-Schmidt as $R(\bx_0+\eps\cdot,
\bx_0+\eps\cdot)$ belongs to $L^2(B\times B)$ (though not necessarily
uniformly in $\eps$; see below) according to proposition
\ref{prop:decompN}. In the same way, we obtain that
$$
\|T^\eps V(\bx_0+\eps \cdot) \|_{H^1(B)} \leq C \|V(\bx_0+\eps \cdot) \|_{L^2(B)},
$$
where $C$ is independent of $\eps$. It remains to show that the
operator $I+\eps^\eta T+\eps^{d-2+\eta}R^\eps$ is injective and to use
the Fredholm alternative to obtain the existence of a unique $w^* \in
H^1(B)$ verifying (\ref{eq:wstar}). Injectivity is obvious when $\eta
\in]0,2]$ since he operator norm of $\eps^\eta T+\eps^{1+\eta} R^\eps$
in $\calL(L^2(B))$ is of order $\calO(\eps^\eta)<1$ for $\eps<\eps_0$
small enough. 

When $\eta=0$, we need to use assumption $\textbf{(H-2)}$.  Since $-1$
is not an eigenvalue of $T$, it suffices to fix $\eps_0$ small enough
such that the distance between $-1$ and the nearest eigenvalue of $T$
is larger than $\eps_0^{d-2} \| R^\eps \|_{\calL(L^2(B))}$. To do so,
we remark, following proposition \ref{prop:decompN}, that uniformly in
$\by \in B$, $R(\cdot,\by) \in W^{2,p}(B)$ with $p < \frac{d}{d-2}$
when $3\leq d\leq 5$ and $p <\infty$ when $d=2$. The Sobolev embedding
then yields that $R(\cdot,\by) \in \calC^0(\overline{B})$ when $d \leq
3$ and $R(\cdot,\by) \in L^{q}(B)$ with $q<\infty$ when $d=4$ and
$q=5$ when $d=5$. Hence,
$$\| R^\eps \|_{\calL(L^2(B))} \leq C\|R (\bx_0+\eps \cdot, \bx_0+\eps
\cdot) \|_{L^2},$$
which is $\calO(1)$ for $d \leq 3$,
$\calO(\eps^{-\alpha})$ for any $\alpha>0$ when $d=4$, and
$\calO(\eps^{-1})$ for $d =5$. For the particular case $q_0 \equiv 0$,
proposition \ref{prop:decompN} gives $R \in
\calC^\infty(\overline{B}\times \overline{B})$ so that $\| R^\eps
\|_{\calL(L^2(B))}$ is bounded independently of $\eps$ for any $d$. In
any event, $\eps_0^{d-2} \| R^\eps \|_{\calL(L^2(B))}=o(\eps_0)$ so
that the Fredholm alternative yields again a unique $w^* \in L^2(B)$
solution to (\ref{eq:wstar}) for $\eps_0$ small enough. In addition,
$w^*$ satisfies the estimate:
\begin{equation} \label{estimwstar}
\| w^* \|_{L^2(B)}\leq C \eps^\eta \|V(\bx_0+\eps \cdot ) \|_{L^2(B)}.
\end{equation}
Then $w^\eps$ is given, for $\by \in \Omega$, by:
$$
 \left\{ 
\begin{array}{l}
\DST w^\eps(\by)=w^*\left(\frac{\by-\bx_0}{\eps}\right), \qquad \by \in \bx_0+\eps B,\\
\DST w^\eps(\by)=\left(-\eps^\eta T w^*-\eps^{d-2+\eta} R^\eps w^*\right)\left(\frac{\by-\bx_0}{\eps}\right) -T^\eps V(\by), \qquad \textrm{otherwise},\\
\end{array} \right.
$$
so that $w^\eps \in H^1(\Omega)$. We verify that $w^\eps$
is then a solution to the variational formulation of
(\ref{eq:wproof}).  To prove uniqueness, we show that, for a
given $u \in H^1(\Omega)$, the assertion
\begin{equation} \label{asser}
\int_\Omega \nabla u \cdot \nabla v \, d\bx +\int_\Omega \left(q_0+\frac{1}{\eps^{2-\eta}} q_1\left(\frac{\cdot-\bx_0}{\eps} \right)\right) \,u \, v d\bx=0, \qquad \forall v\in H^1(\Omega),
\end{equation}
implies $u=0$. Indeed, for $\varphi \in L^2(\Omega)$, consider the
weak solution $v \in H^1(\Omega)$ of
 $$-\Delta v+q_0 v=\varphi, \quad  \bx \in \Omega,
 $$
 augmented with homogeneous Neumann conditions on $\partial
 \Omega$. Thus, $v$ is given by $v(\by)=\int_{\Omega} N(\bx,\by)
 \varphi(\bx) d\bx$. Plugging $v$ into (\ref{asser}) leads to
$$
\int_\Omega (u+T^\eps u) \varphi d\bx=0,\qquad\forall\varphi \in L^2(\Omega),
$$
so that $u+T^\eps u=0$, which implies that $u=0$. This ends the
proof of existence of a unique solution of the variational formulation
of (\ref{eq:wproof}) when $d \geq 3$. 

We treat now the case $d=2$. When $\eta>0$, existence and uniqueness
can be established in the same manner as above.  When $\eta=0$, we use
assumption \textbf{(H-3)}. We first notice that for $d=2$, we have
$$
T^\eps w^*(\bx_0+\eps \by)=\int_{B} q_1\left(\bx \right) w^*(\bx) \Gamma(\bx- \by) d\by-\frac{\log \eps}{2\pi} \int_{B} q_1\left(\bx \right) w^*(\bx) d\bx+R^\eps w^*(\by).
$$
In the same way, proposition $\ref{prop:decompN}$ gives, uniformly in $\by \in B$, $R(\cdot,\by) \in W^{2,p}(B) \subset \calC^1(\overline{B})$ with $p<\infty$, so that we can recast $R^\eps w^*$ as 
\begin{eqnarray*}
R^\eps w^*(\by)&=&R(\bx_0, \bx_0+\eps \by) \int_{B} q_1\left(\bx \right) w^*(\bx) d\bx  + \widetilde{R}^\eps w^*(\by) \\
\widetilde{R}^\eps w^*(\by)&=&\int_{B} q_1\left(\bx \right) w^*(\bx) \left( R(\bx_0+\eps \bx , \bx_0+\eps \by)-R(\bx_0, \bx_0+\eps \by)\right) d\bx.
\end{eqnarray*}
The system (\ref{eq:wstar}) can then be reformulated as:
$$
w^*+T w^*+\widetilde{R}^\eps w^*=-T V-\widetilde{R}^\eps V-\left(\frac{\log \eps}{2\pi} + R(\bx_0, \bx_0+\eps \cdot)\right) C^\eps,
$$
where the constant $C^\eps$ is equal to
$$C^\eps=\int_{B}q_1\left(\bx \right) \left(V(\bx_0+\eps \bx)+w^*(\bx) \right)d\bx=\int_{B}q_1\left(\bx \right) v^\eps(\bx_0+\eps \bx) d\bx.
$$
Under assumption \textbf{(H-3)}, we have $C^\eps=0$ so that we just
need to show that
\begin{displaymath}
  \| \widetilde{R}^\eps \|_{\calL(L^2(B))}=o(1)
\end{displaymath}
to apply the Fredholm alternative. Since $R(\cdot,\by)\in
\calC^{1}(\overline{B})$, uniformly in $\by$, we have, for all
$(\bx,\by) \in \overline{B} \times \overline{B}$, that
\begin{math}
   |R(\bx_0+\eps \bx , \bx_0+\eps \by)-R(\bx_0, \bx_0+\eps \by)| \leq C \eps,
\end{math}
which gives $\| \widetilde{R}^\eps
\|_{\calL(L^2(B))}=\calO(\eps)$ and ends the proof of existence when
$d=2$.\\

We now prove decomposition (\ref{decompvtheo}), which is the corner
stone of the proof of theorem \ref{th:asympV}.  Since
$w^\eps(\bx_0+\eps \bx)=w^*(\bx)$ when $\bx \in B$, it suffices to obtain
an expression for $w^*$. We consider first the case $d\geq3$. Defining
$V^\eps(\bx):=V(\bx_0+\eps \bx)$, we recast (\ref{eq:wstar}) as:
$$
w^*+\eps^\eta T w^*=-\eps^\eta T V^\eps-\eps^{d-2+\eta}R^\eps \left(V^\eps+w^* \right).
$$
We expand $V^\eps$ in the first term of the right hand side and set
$w^*=\eps^\eta \Psi^\eps+\eps^{d-2+\eta}\,r^\eps+r_V^\eps$, so as to
obtain:
\begin{eqnarray*} 
\Psi^\eps+\eps^{\eta} T \Psi^\eps &=& -T  \sum_{|j|=0}^{d+1} \frac{\eps^{|j|}}{j!}  \bx^j \partial^j V (\bx_0),\\
r^\eps+\eps^{\eta} T r^\eps &=& -R^\eps \left(V^\eps+w^* \right), 
\hspace{2cm}
r_V^\eps+\eps^{\eta} T r_V^\eps\,\,\,=\,\,\, -T R_V^\eps,
\end{eqnarray*}
where $R_V^\eps$ is the remainder of the Taylor expansion of $V^\eps\in
\calC^\infty(\overline{B})$ of order $d+2$.  Writing
$\Psi^\eps(\bx):=\sum_{|j|=0}^{d+1}\frac{\eps^{|j|}}{j!}  \partial^j
V(\bx_0) \phi^\eta_j(\bx)$, with
$$
\phi^\eta_j(\bx)+\eps^{\eta} T \phi^\eta_j(\bx) = -T \bx^j,
$$
and following the preceding proof of existence when $d\geq 3$, we
verify that $r^\eps_V \in H^1(B)$ with a norm bounded by $C
\eps^{d+2}$ and that $r^\eps$ and $\phi^\eta_j$ are uniquely defined
in $H^1(B)$. Also, examining $\|R^\eps \|_{\calL(L^2(B))}$ as in the
proof of existence, we find that $r^\eps$ is bounded in $L^2(B)$
independently of $\eps$ when $d =3$, is $\calO(\eps^{-\alpha})$ for
any $\alpha>0$ when $d=4$ and $\calO(\eps^{-1})$ when $d=5$. When $q_0
\equiv 0$, $r^\eps$ is bounded in $H^1(B)$ independently of $\eps$
since $\|R^\eps \|_{\calL(H^1(B))}$ is uniformly bounded. We thus
obtain the expression (\ref{decompvtheo}) announced in the proposition
for $d \geq 3$.  When $d=2$, the equation for $r^\eps$ has to be
replaced by
\begin{eqnarray*}
r^\eps+\eps^{\eta} T r^\eps &=& -R^\eps \left(V^\eps+w^* \right)+\frac{\log \eps}{2\pi}\int_{B}q_1\left(\bx \right) v^\eps(\bx_0+\eps \bx) d\bx,\\
&=&-\widetilde{R}^\eps  \left(V^\eps+w^* \right)-\left(\frac{\log \eps}{2\pi} + R(\bx_0, \bx_0+\eps \cdot)\right) C^\eps,
\end{eqnarray*}
where $C^\eps$ is the same constant as before. When $\eta>0$, we
verify that $r^\eps \in H^1(B)$ with a norm of order $\log
\eps$. When $\eta=0$, assumption \textbf{(H-3)} implies $C^\eps=0$.
Since $\widetilde{R}^\eps$ is $\calO(\eps)$ in
$\calL(L^2(B),H^1(B))$, we deduce that $r^\eps$ is $\calO(\eps)$ in
$H^1(B)$ since $V$ is uniformly bounded in $B$ and $w^*$ is bounded in
$L^2(B)$ according to (\ref{estimwstar}).
\end{proofof}
\bigskip

\begin{proofof} {\em Theorem \ref{th:asympV}.}
  We express $v^\eps$ in terms of $V$ and the Green function
  $N$, to obtain, \textit{a.e.} in $\Omega$:
\begin{equation} \label{eq:vproof}
v^\eps(\by)=V(\by)-\eps^{d-2+\eta}\int_B q_1(\bx)\left(V(\bx_0+\eps \bx)+w^\eps(\bx_0+\eps \bx) \right) N(\bx_0+\eps \bx, \by) d\bx.
\end{equation}
Taking the trace of (\ref{eq:vproof}) on $\partial \Omega$, which is
well defined in $L^2(\partial \Omega)$ and thus almost everywhere,
replacing $w^\eps(\bx_0+\eps \bx)$ by the expression in
(\ref{decompvtheo}) and Taylor expanding both $V$ and $N$ according to
(\ref{decompNhelm}), lead to the result.
\end{proofof} 
\bigskip

\begin{proofof}{\em Proposition \ref{prop:equiv}.}
  The outline of the proof is as follows: starting from the asymptotic
  expansion for $v^\eps$ in theorem \ref{th:asympV}, our aim is to
  recover that of $u^\eps$ in theorem \ref{th:asymp} and the
  expression of the polarization tensor $M$. This is done in several
  steps. First, we verify that assumptions $\textbf{(H-1)}$,
  $\textbf{(H-2)}$ and $\textbf{(H-3)}$ are satisfied for the
  particular form (\ref{eq:defq1}) of the potential $q_1$. In a second
  step, we show that the term $f^\eps$ in the expansion of $v^\eps$ is
  of order $\calO(\eps^4)$ so that $\eps^{2(d-2)} f^\eps$ is
  $\calO(\eps^{2d})$ and can be treated as a remainder. Then, we show
  in (\ref{eq:zero1})--(\ref{eq:zero2}) that the two first-order terms
  in the expansion of $v^\eps$ are actually of order
  $\calO(\eps^{2d})$ so that they can be neglected and the expansions
  for $v^\eps$ and $u^\eps$ have the same leading order
  $\calO(\eps^{d})$. Finally, using the particular form of the
  potential $q_1$, we perform some transformations in the polarization
  tensors $Q$ and $Q^\eta$ for $\eta=0$ leading to the expression of
  the polarization tensor $M$ in theorem \ref{th:asymp}.
  
  We will need the following lemma, which is one of the main
  ingredients to show the equivalence of the tensors:
\begin{lemma} \label{integpart}
  Assume $v \in H^1(B)$ verifies in the distribution sense,
\begin{equation} \label{disv}
-\Delta v+q_1v=h \quad \textrm{in } \calD'(B),\qquad
  q_1(\bx)=\frac{\Delta \sqrt{D_0+D_1(\bx)}}{\sqrt{D_0+D_1(\bx)}},\qquad
  h\in L^2(B).
\end{equation} 
Then, for all $\varphi \in H^1(B)$ and harmonic in $B$, we have:
$$
\int_B q_1(\bx)\, v(\bx) \varphi(\bx)d\bx
=\frac{1}{\sqrt{D_0}} \int_B D_1(\bx) \nabla \Big( \frac{v(\bx)}{\sqrt{D_0+D_1(\bx)}} \Big) \cdot \nabla  \varphi(\bx)d\bx- \int_{B} h \varphi d\bx. 
$$
\end{lemma}
\begin{proof} 
  Define $D(\bx):=D_0+D_1(\bx)$. Note that $\partial_\bn D_1=0$ on
  $\partial B$ since $D_1 \in \calC^2(\Omega)$ and $D_1$ is supported
  in $B$. Hence, two successive integrations by parts yield:
  \begin{displaymath}
    \int_B \frac{\Delta \sqrt{D(\bx)}}{\sqrt{D(\bx)}} v(\bx) \varphi(\bx)d\bx
=-\int_{B} \nabla \sqrt{D} \cdot \nabla \Big(\frac{v \varphi} {\sqrt{D}} \Big) d\bx
=\int_{B} \Big(\sqrt{D}-\sqrt{D_0}\Big) \Delta \Big(\frac{v \varphi} {\sqrt{D}} \Big) d\bx.
  \end{displaymath}  
  The above expression makes sense since $\varphi$ is harmonic and
  $\Delta v \in L^2(B)$ because of (\ref{disv}). Starting from
  (\ref{disv}), we verify after some algebra that $v$ solves
\begin{equation} \label{eq:vdiff}
\nabla \cdot D \nabla \Big(\frac{v} {\sqrt{D}}\Big)=\sqrt{D}h, \quad \textrm{in } \calD'(B),
\end{equation}
which, since $D>0$ in $\Rm^d$, is equivalent to: 
$$
2 \, \nabla \sqrt{D} \cdot \nabla \Big(\frac{v}{\sqrt{D}} \Big)+\sqrt{D} \Delta \Big(\frac{v} {\sqrt{D}}\Big)=h.
$$
Since $D=D_0$ on $\partial B$ and is constant, it follows from the
above equation and another integration by parts that:
\begin{equation} \label{eq:intpart}
\DST 2 \int_{\partial B} \frac{\partial v}{\partial \bn } \varphi d\sigma -\int_{B} \sqrt{D} \Delta \Big(\frac{v}{\sqrt{D}} \Big) \varphi d\bx \DST -2 \int_{B} \sqrt{D}   \nabla \Big(\frac{v}{\sqrt{D}} \Big) \cdot \nabla \varphi d\bx=
 \int_{B} h \varphi d\bx.
\end{equation}
Here, $\sigma$ is the surface measure on $\partial B$ and the boundary
term above has to be understood as the $H^{-\frac12}(\partial
B)$--$H^{\frac12}(\partial B)$ duality product since $\partial_\bn v
\in H^{-\frac12}(\partial B)$ because $v \in H^1(B)$ and $\Delta v \in
L^2(B)$ thanks to (\ref{disv}). Using the fact that $\varphi$ is
harmonic in $B$, that $D=D_0$ on $\partial B$, and using equation
(\ref{eq:intpart}), we find:
\begin{eqnarray*}
&&\int_{B} \Big(\sqrt{D}-\sqrt{D_0}\Big) \Delta \Big(\frac{v \varphi} {\sqrt{D}} \Big) d\bx\\
&=&\int_{B} \Big(\sqrt{D}-\sqrt{D_0}\Big) \Delta \Big(\frac{v} {\sqrt{D}} \Big) \varphi \,d\bx
+2 \int_{B} \Big(\sqrt{D}-\sqrt{D_0}\Big) \nabla \Big(\frac{v} {\sqrt{D}} \Big) \cdot \nabla \varphi \,d\bx,\\
&=&-\int_{B} \!\!\! \sqrt{D_0}\,\Delta \Big(\frac{v} {\sqrt{D}} \Big) \varphi \,d\bx
-2 \int_{B}  \!\!\! \sqrt{D_0}\,\nabla \Big(\frac{v} {\sqrt{D}} \Big) \cdot \nabla \varphi \,d\bx
+2 \int_{\partial B} \frac{\partial v}{\partial \bn } \varphi d\sigma- \int_{B}  \!\!\! h \varphi d\bx,\\
&=&-\int_{B} \sqrt{D_0}\,\nabla \Big(\frac{v} {\sqrt{D}} \Big) \cdot \nabla \varphi \,d\bx
+\int_{\partial B} \frac{\partial v}{\partial \bn } \varphi d\sigma- \int_{B} h \varphi d\bx.
\end{eqnarray*}
To conclude, we just need to remark that, thanks to (\ref{eq:vdiff}),
$$
\int_{\partial B} \frac{\partial v}{\partial \bn } \varphi d\sigma=\frac{1}{\sqrt{D_0}}\int_{\partial B} D \frac{\partial }{\partial \bn } \Big( \frac{v} {\sqrt{D}}\Big) \varphi d\sigma
=\frac{1}{\sqrt{D_0}} \int_{B} D \nabla \Big(\frac{v} {\sqrt{D}} \Big) \cdot \nabla \varphi d\bx.
$$
\end{proof}
Coming back to the proof of proposition \ref{prop:equiv}, we first
verify that assumptions $\textbf{(H-1)}$, $\textbf{(H-2)}$ and
$\textbf{(H-3)}$ are satisfied. Since $q_0=0$, $\textbf{(H-1)}$
trivially holds because of the compatibility conditions
(\ref{eq:norma}). The same is true for $\textbf{(H-3)}$. Regarding
$\textbf{(H-2)}$, we have to show that if
\begin{equation} \label{injec}
\varphi+T\varphi=0, \qquad \forall \varphi \in L^2(B),
\end{equation}
then $\varphi=0$. To this aim, we first remark that $T$ maps $L^2(B)$
to $H^1(B)$, so that every $\varphi$ verifying (\ref{injec}) belongs
to $H^1(B)$. Now, $\varphi$ can be extended to $\Rm^d$ to a function
$\varphi^* \in H^1_{\textrm{loc}}(\Rm^d)$ by the relation:
$$
 \left\{ 
\begin{array}{l}
\DST \varphi^*(\by)=\varphi \left(\by \right), \qquad \by \in B,\\
\DST \varphi^*(\by)= -T\varphi(\by), \qquad \textrm{otherwise}.\\
\end{array} \right.
$$
Moreover, when (\ref{injec}) holds, then so does the following in
the distributional sense:
\begin{equation} \label{distribvarphi}
-\Delta \varphi^* +q_1\varphi^*=0, \quad \textrm{in } \calD'(\Omega'),
\end{equation}
for any bounded set $\Omega' \subset \Rm^d$. Consider $\by \in
\Rm^d\backslash \overline{B}$. Then $\Gamma(\bx-\by)$ is harmonic for
$\bx \in B$. We then apply lemma \ref{integpart} with $h=0$ to find,
uniformly in $\by$:
\begin{eqnarray*}
\varphi^*(\by)&=&-\int_{B} q_1\left(\bx \right) \varphi^*(\bx) \Gamma(\bx-\by) d\by,\\
&=&-\frac{1}{\sqrt{D_0}} \int_B D_1(\bx) \nabla \left( \frac{\varphi^*(\bx)}{\sqrt{D_0+D_1(\bx)}} \right) \cdot \nabla \Gamma(\bx-\by)d\bx.
\end{eqnarray*}
We thus deduce from the above equation for $d\geq 2$ the following
behavior at infinity:
\begin{equation} \label{phiinfty}
\varphi^*(\by)=\calO(|\by|^{1-d}) , \qquad  \nabla \varphi^*(\by)=\calO(|\by|^{-d}).
\end{equation}
Besides, equation (\ref{distribvarphi}) can be reformulated as:
\begin{equation} \label{diffvarphi}
\nabla \cdot (D_0+D_1) \nabla  \left(\frac{\varphi^*}{\sqrt{D_0+D_1}} \right)=0, \quad \textrm{in } \calD'(\Omega').
\end{equation} 
After multiplication by $\overline{\varphi^*}(D_0+D_1)^{-\frac12}$ in
$H^1_{\rm loc}(\Rm^d)$, and an integration on the sphere $B_R \supset
\supset B$ of radius $R$ and boundary $S_R$, we find:
$$
\int_{B_R} (D_0+D_1)\left| \nabla  \left(\frac{\varphi^*}{\sqrt{D_0+D_1}} \right)\right|^2 d\bx-\int_{S_R} \frac{\partial \varphi^*}{\partial \bn} \overline{ \varphi^*} d\sigma=0.
$$
Letting $R \to \infty$ leads, together with (\ref{phiinfty}), to
$\varphi=0$ so that assumption $\textbf{(H-2)}$ is satisfied.

We now show the equivalence of the tensors. First, the term $f^\eps$
given in the expansion of theorem \ref{th:asympV} is of order
$\calO(\eps^4)$, which is not obvious at first sight. Consequently,
$\eps^{2(d-2)}f(\eps)$ is of order $\calO(\eps^{2d})$ and can treated
as a remainder in the expansion. To prove this, we apply lemma
\ref{integpart} to $f^\eps$ and need to estimate $r^\eps$. Let us
recall the equation verified by $r^\eps \in H^1(B)$ given in
proposition \ref{prop:asympV}:
\begin{equation} \label{rint}
r^\eps(\by)+T r^\eps(\by) = \int_{B} q_1\left(\bx \right) v^\eps(\bx_0+\eps \bx) R(\bx_0+\eps \bx, \bx_0+\eps \by)d\bx.
\end{equation}
When $d=2$, we use the fact that assumption $\textbf{(H-3)}$ is
satisfied since $q_0=0$ so that the term involving $\log \eps$ in the
equation of proposition \ref{prop:asympV} vanishes. Since $v^\eps$
verifies (\ref{disv}) with $h=0$, and $R$ verifies $(\ref{eq:R1})$
with $q_0=0$ so that we have $\Delta_\by R(\bx,\by)=\Delta_\by
R(\by,\bx)=0$ since $R$ is symmetric in its arguments and is thus
harmonic, we apply lemma \ref{integpart} to find:
$$
\begin{array}{l}
\DST \int_{B} q_1\left(\bx \right) v^\eps(\bx_0+\eps \bx) R(\bx_0+\eps \bx, \bx_0+\eps \by)d\bx\\
\DST \hspace{2cm}=\frac{\eps}{\sqrt{D_0}}\int_{B} D_1(\bx) \nabla \left(\frac{v^\eps(\bx_0+\eps \bx)}{\sqrt{D_0+D_1(\bx)}}  \right)\cdot \nabla_\bx R(\bx_0+\eps \bx, \bx_0+\eps \by) d\bx.
\end{array}
$$
Moreover, we show that
\begin{equation} \label{vO1}
\left\| \nabla \left(\frac{v^\eps(\bx_0+\eps \cdot)}{\sqrt{D_0+D_1}} \right) \right\|_{L^2(B)} = \calO(\eps),
\end{equation}
so that the left hand side of $(\ref{rint})$ is of order
$\calO(\eps^2)$.  This is obtained by proving that the leading term in
the above expression vanishes. That is to say, thanks to the
decomposition given theorem \ref{th:asympV},
\begin{math}
  v^\eps(\bx_0+\eps \by)=V(\bx_0+\eps
  \by)+\sum_{|j|=0}^{d+1}\frac{\eps^{|j|}}{j!} \partial^j V(\bx_0)
  \phi_j(\by)+\eps^{d-2}\,r^\eps(\by)+\calO(\eps^{d+2}), \,\by \textit{
    a.e.} \textrm{ in } B,
\end{math}
that
\begin{equation} \label{firstorder}
\nabla \Big(\frac{V(\bx_0)(1+\phi_0(\bx))}{\sqrt{D_0+D_1(\bx)}} \Big)=0.
\end{equation}
The argument is very similar to that in the verification of assumption
$\textbf{(H-2)}$ and so we just sketch the proof. Since $\phi_0$
verifies
\begin{math}
  \phi_0+ T \phi_0 = -T 1,
\end{math}
it can be extended to $\Rm^d$ to $\phi_0^* \in
H^1_\textrm{loc}(\Rm^d)$ which admits the behavior at infinity given
in (\ref{phiinfty}). We also have, for any bounded set $\Omega'
\subset \Rm^d$,
\begin{equation} \label{phi+1}
-\Delta (\phi_0^*+1) +q_1 (\phi_0^*+1)=0, \quad \textrm{in } \calD'(\Omega'),
\end{equation}
so that, still denoting by $B_R$ the sphere of radius $R$, 
$$
\int_{B_R} (D_0+D_1)\left| \nabla  \Big(\frac{\phi_0^*+1}{\sqrt{D_0+D_1}} \Big)\right|^2 d\bx-\int_{S_R} \frac{\partial \phi^*_0}{\partial \bn} (\overline{ \varphi_0^*}+1) d\sigma=0.
$$
Sending $R$ to infinity then gives the result thanks to the decay
of $\nabla \phi_0^*$ at infinity.  Owing to this result, the
decomposition (\ref{decompvtheo}), the fact that $\phi_j$ and $r^\eps$
belong to $H^1(B)$, and $r^\eps$ is at least an $\calO(\eps)$ when
$d=2$ as mentioned in theorem (\ref{th:asympV}), we get that
(\ref{vO1}) holds.  Furthermore, using again the fact that $R$ is
harmonic, we verify from (\ref{rint}) that $r^\eps$ solves in the
distribution sense:
$$ 
-\Delta r^\eps +q_1 r^\eps=0, \quad \textrm{in } \calD'(B). 
$$
We cannot apply lemma \ref{integpart} directly to (\ref{rint})
since for $(\bx,\by) \in B \times B$, we have
$$
-\Delta_\bx \Gamma(\bx-\by)=\delta(\bx-\by), \quad \textrm{in } \calD'(B),
$$
and $\Gamma$ is not harmonic. Nevertheless, the lemma can easily be
adapted to this special case so that, $\by$ \textit{a.e.} in $B$, we have
\begin{eqnarray*}
\int_B q_1(\bx)\, r^\eps(\bx) \Gamma(\bx-\by)d\bx
&=&\frac{1}{\sqrt{D_0}} \int_B D_1(\bx) \nabla \Big( \frac{r^\eps(\bx)}{\sqrt{D_0+D_1(\bx)}} \Big) \cdot \nabla \Gamma(\bx-\by)  d\bx\\
&&-\frac{\sqrt{D_0+D_1(\by)}-\sqrt{D_0}}{\sqrt{D_0+D_1(\by)}} r^\eps(\by).
\end{eqnarray*}
Plugging the above expression into (\ref{rint}), we finally find the
following equation for $r^\eps \in H^1(B)$, $\by$ \textit{a.e.} in
$B$:
$$
\frac{r^\eps(\by)}{\sqrt{D_0+D_1(\by)}}+\frac{1}{D_0} \int_B D_1(\bx) \nabla \Big( \frac{r^\eps(\bx)}{\sqrt{D_0+D_1(\bx)}} \Big) \cdot \nabla \Gamma(\bx-\by)  d\bx
$$
$$
=\frac{\eps }{D_0} \int_B D_1(\bx) \nabla \Big( \frac{v^\eps(\bx)}{\sqrt{D_0+D_1(\bx)}} \Big) \cdot \nabla_\bx R(\bx_0+\eps \bx, \bx_0+\eps \by)   d\bx.
$$
Identifying the right hand side of the latter equation with
$S^\eps(\eps \bx)$ and $(D_0+D_1)^{-\frac12}r^\eps$ with
$r^\eps_1(\bx)+S^\eps(\eps \bx)$ in the proof of theorem
\ref{th:asymp}, we see that $(D_0+D_1)^{-\frac12} r^\eps$ and
$r^\eps_1(\bx)+S^\eps(\eps \bx)$ satisfy similar equations so that the
same technique yield
$$
\left\| \nabla \Big(\frac{r^\eps}{\sqrt{D_0+D_1}} \Big) \right\|_{L^2(B)} \leq C\eps^2 \left\| \nabla \Big(\frac{v^\eps}{\sqrt{D_0+D_1}} \Big) \right\|_{L^2(B)} \left\| \nabla_\bx  \nabla_\by R \right\|_{L^\infty(B_0 \times B_0)}.
$$
$\mbox{}$From proposition \ref{prop:decompN}, $R \in \calC^\infty(\Omega
\times \Omega)$. Together with (\ref{vO1}), this finally gives that:
\begin{equation} \label{rO3}
\left\| \nabla \Big(\frac{r^\eps}{\sqrt{D_0+D_1}} \Big) \right\|_{L^2(B)} = \calO(\eps^3).
\end{equation}
We conclude by applying once again lemma \ref{integpart} to obtain  
\begin{eqnarray*}
\| f^\eps\|_{L^2(\partial \Omega)}&=&  \left\| \int_B q_1(\bx) r^\eps (\bx) N(\bx_0+\eps \bx, \cdot) d\bx \right\|_{L^2(\partial \Omega)}\\
&=&\frac{\eps}{\sqrt{D_0}} \left\|\int_{B} D_1(\bx) \nabla \Big(\frac{r^\eps(\bx)}{\sqrt{D_0+D_1(\bx)}}  \Big)\cdot \nabla_\bx N(\bx_0+\eps \bx, \cdot) d\bx \right\|_{L^2(\partial \Omega)}\!\!\!
=\calO(\eps^4),
\end{eqnarray*}
thanks to (\ref{rO3}).

We now prove (\ref{eq:zero1}) and (\ref{eq:zero2}) so that the leading
order in the expansion of theorem \ref{th:asymp} is $\calO(\eps^d)$ as
in the case of the diffusion equation. We remark that, for $\eta=0$,
\begin{eqnarray*}
\sum_{j=0}^{d+1} \frac{\eps^{|j|}}{j!} \partial^j V (\bx_0)\Big(Q_{0j}+Q_{0j}^0\Big)&=&\sum_{j=0}^{d+1} \frac{\eps^{|j|}}{j!} \partial^j V (\bx_0)\int_B q_1(\bx) 
\Big( \phi^\eta_j(\bx)+\bx^j\Big) d\bx,\\
&=&\int_B q_1(\bx) 
\Big( \Psi^\eps(\bx) + V(\bx_0+\eps \bx)-R_V^\eps(\bx)\Big) d\bx,
\end{eqnarray*}
where $\Psi^\eps$ is given in the theorem and $R_V^\eps$ is the
remainder of the Taylor expansion of $V(\bx_0+\eps \bx)$ at the order
$d+2$ and is thus of order $\calO(\eps^{d+2})$. In order to apply
lemma \ref{integpart}, we verify from (\ref{eqphith2}) that $\Psi^\eps
\in H^1(B)$ solves,
$$
-\Delta \Psi^\eps +q_1 \Psi^\eps=-q_1 \Big(V(\bx_0+\eps \bx)-R_V^\eps(\bx) \Big)\quad \textrm{in } \calD'(B).
$$
Setting $v(\bx)=\Psi^\eps(\bx)+V(\bx_0+\eps \bx)$, $h=q_1 R_V^\eps $ and $\varphi=1$ in lemma \ref{integpart} yields (\ref{eq:zero1}). Regarding (\ref{eq:zero2}), we write, for $\by \in \partial \Omega$,
$$
\sum_{i=0}^{d+1} \frac{\eps^{|i|}}{i!} \partial^i N(\bx_0,\by)\Big(Q_{i0}+Q_{i0}^0\Big)=\int_B q_1(\bx) 
\Big( 1+\phi_0(\bx)\Big) (N(\bx_0+\eps \bx,\by)-R_N^\eps(\bx,\by)) d\bx,
$$
where $R_N^\eps$ is the remainder of the $d+2$ order Taylor
expansion of $N(\bx_0+\eps \bx,\by)$ with respect to $\bx$ and is thus
of order $\calO(\eps^{d+2})$. Since $N(\bx_0+\eps \bx,\by)$ is
harmonic when $\bx\in B$ and $\by\in\partial \Omega$, we apply lemma
\ref{integpart} thanks to (\ref{phi+1}) to find:
\begin{eqnarray*}
\sum_{i=0}^{d+1} \frac{\eps^{|i|}}{i!} \partial^i N(\bx_0,\by)\Big(Q_{i0}+Q_{i0}^0\Big)&=&\frac{1}{\sqrt{D_0}}\int_B D_1\nabla \Big(\frac{(1+\phi_0(\bx))}{\sqrt{D_0+D_1(\bx)}} \Big) \cdot \nabla_\bx N(\bx_0+\eps \bx,\by)d\bx\\
&&+\calO(\eps^{d+2}) =\calO(\eps^{d+2}),
\end{eqnarray*}
since the above integral vanishes thanks to (\ref{firstorder}).  

At this point of the proof, we have thus shown that $v^\eps$
satisfies, \textit{a.e.} on $\partial \Omega$, that
$$ 
\left. v^\eps(\by) \right|_{\partial \Omega}=\left. V(\by)\right|_{\partial \Omega}-\sum_{|j|=1}^{d+1} \sum_{|i|=1}^{d+1} \frac{\eps^{d-2+|i|+|j|}}{i!j!}  \left(Q_{ij}+Q^0_{ij} \right) \partial^j V (\bx_0) \left. \partial^i N(\bx_0,\by)\right|_{\partial \Omega}
+\calO(\eps^{2d}).
$$
Setting $v^\eps(\by):= u^\eps(\by)
\sqrt{D_0+D_1(\frac{\by-\bx_0}{\eps})}$, $V:=\sqrt{D_0} U$, we verify
that $u^\eps$ and $U$ are solutions to (\ref{eq:diff}) and
(\ref{eq:diffD0}), respectively, with the boundary term $g$ multiplied
by $\sqrt{D_0}$. It thus remains to show that (\ref{equiv1}) and
(\ref{equiv2}) hold to recover the asymptotic expansion for $u^\eps$
of theorem \ref{th:asymp}.  Since $\bx^j+\phi_j$ satisfies
(\ref{disv}) when $|j|=1$ and $\bx^i$ is harmonic when $|i|=1$, we
have, for $|i|=|j|=1$:
\begin{eqnarray}
Q_{ij}+Q_{ij}^0&=&\int_B q_1(\bx) \Big( \bx^j+\phi_j(\bx)\Big)\bx^i d\bx, \nonumber\\
&=&\frac{1}{\sqrt{D_0}}\int_B D_1\nabla \Big(\frac{\bx^j+\phi_j(\bx)}{\sqrt{D_0+D_1(\bx)}} \Big) \cdot \nabla \bx^id\bx \label{eqQ1}.
\end{eqnarray}
We introduce the following
extension to $\phi_j$ on $\Rm^d$:
$$
 \left\{ 
\begin{array}{l}
\DST \phi_j^*(\by)=\phi_j \Big(\by \Big), \qquad \by \in B,\\
\DST \phi_j^*(\by)= -T\phi_j(\by)-T\bx^j, \qquad \textrm{otherwise},\\
\end{array} \right. 
$$
which thus satisfies the conditions at infinity in
(\ref{phiinfty}).  We recall that $\phi_{j0}^0$, the function
introduced in theorem \ref{th:asymp} to define the polarization tensor
$M$, is the unique weak solution in the space $H^1_\textrm{loc}(\Rm^d)
\cap \calC^\infty(\Rm^d \backslash \overline{B})$ to the following system posed
in $\Rm^d$:
\begin{eqnarray}
\nabla \cdot \Big(D_0+D_1(\bx) \Big)\nabla \phi_{j0}^0&=&-\nabla \cdot \Big(D_1(\bx) \nabla \bx^j \Big), \label{eqphi1}\\
\phi_{j0}^0(\bx) &=& \calO(|\bx|^{1-d})\quad \textrm{as} \quad |\bx|\to \infty.\label{eqphi2}
\end{eqnarray}
When $|j|=1$, notice that $\phi_{j0}^0$ is given by
$$
 \phi_{j0}^0(\bx)=\Big(\frac{\sqrt{D_0+D_1}}{\sqrt{D_0}}-1\Big)\bx^j\, +\frac{\sqrt{D_0+D_1}}{\sqrt{D_0}}\phi_j^*(\bx), 
 $$
 so that (\ref{equiv1}) is proved using (\ref{eqQ1}). To prove
 (\ref{equiv2}), we need to sum over $i$ and $j$ to be able to use
 lemma \ref{integpart} since $\bx^i$ is not harmonic for $|i|\geq 2$
 and $\bx^j+\phi_j(\bx)$ satisfies (\ref{disv}) with a negligible
 left- hand side $h$ of order $\calO(\eps^{d+2})$ only after
 summation.  We thus write, using the same arguments as for the proof
 of (\ref{eq:zero1}) and (\ref{eq:zero2}), for $\by \in \partial
 \Omega$:
$$
\begin{array}{l}
\hspace{1cm}
\DST \sum_{|j|=1}^{d+1} \sum_{|i|=1}^{d+1} \frac{\eps^{|i|+|j|}}{i!j!}  \Big(Q_{ij}+Q^0_{ij} \Big)\partial^i N(\bx_0,\by) \partial^j V (\bx_0)\\[4mm]
\DST 
=\sum_{|j|=1}^{d+1} \sum_{|i|=1}^{d+1} \frac{\eps^{|i|+|j|}}{i!j!}\partial^i N(\bx_0,\by) \partial^j V (\bx_0) \int_B q_1(\bx) 
\Big( \bx^j+\phi_j(\bx)\Big) \bx^i d\bx,\\[4mm]
\DST 
=\int_B \!q_1(\bx) \big(V(\bx_0+\eps \bx)-R_V^\eps(\bx)+\Psi^\eps(\bx)\big)(N(\bx_0+\eps \bx,\by)-R_N^\eps(\bx,\by)) d\bx +\calO(\eps^{d+2}),\\[4mm]
\DST 
=\frac{\eps}{\sqrt{D_0}}\int_B D_1\nabla \Big(\frac{V(\bx_0+\eps \bx)+\Psi^\eps(\bx)}{\sqrt{D_0+D_1(\bx)}} \Big) \cdot \nabla_\bx N(\bx_0+\eps \bx,\by)d\bx+\calO(\eps^{d+2}),\\[4mm]
\DST  
=\sum_{|j|=1}^{d+1} \sum_{|i|=1}^{d+1}
\frac{\eps^{|i|+|j|}}{i!j!} \partial^i N(\bx_0,\by) \partial^j V (\bx_0) 
 \int_B D_1(\bx) \nabla \Big(\frac{\bx^j+\phi_j(\bx)}{\sqrt{D_0+D_1(\bx)}}  \Big) \! \cdot\! \nabla \bx^i d\bx 
+\calO(\eps^{d+2}).
\end{array}
$$
It remains to relate the latter sum to $M$. For that, let $f_j$ be defined as:
$$
f_{j}(\by)=\Big(\frac{\sqrt{D_0+D_1}}{\sqrt{D_0}}-1\Big)\bx^j\, +\frac{\sqrt{D_0+D_1}}{\sqrt{D_0}}\phi_j^*(\by)- \phi_0^0 (\by ).
$$
Then $f_j$ belongs to $H^1_\textrm{loc}(\Rm^d) \cap
\calC^\infty(\Rm^d \backslash \overline{B})$ by construction and
solves:
\begin{eqnarray} 
\nabla \cdot (D_0+D_1(\bx)) \nabla f_j &=&-\un_B(\bx) \sqrt{D_0+D_1(\bx)}\,\Delta \bx^j, \quad \bx \in \Rm^d, \label{eqf1}\\
f_j(\by)&=&\calO(|\by|^{1-d}) \qquad \textrm{ as }\quad |\by|\to \infty.\label{eqf2}
\end{eqnarray}
Here, $\un_B$ is the characteristic function of the set $B$ and
$\phi_j^*$ is the extension of $\phi_j$ to $\Rm^d$. Note that $f_j=0$
when $|j|=1$ so that we recover the preceding relationship between
$\phi_j^*$ and $\phi_0^0$. To conclude the proof, it suffices to show
that an appropriate linear combination of the terms $f_j$ is of order
$\calO(\eps^{d+2})$. Let:
\begin{displaymath}
  T^\eps_V(\bx):=\sum_{|j|=1}^{d+1} \frac{\eps^{|j|}}{j!}\partial^j V(\bx_0)  \Delta \bx^j,\quad
F^\eps(\bx):=\sum_{|j|=1}^{d+1} \frac{\eps^{|j|}}{j!}\partial^j V(\bx_0)  f_j(\bx),
\end{displaymath}
so that since $\Delta V(\bx_0+\eps \bx)=0$, for all $\bx \in B$, we
have $T^\eps_V(\bx)=\calO(\eps^{d+2})$ uniformly in $B$ and $F^\eps
\in H^1_\textrm{loc}(\Rm^d) \cap \calC^\infty(\Rm^d \backslash
\overline{B})$ solves
\begin{eqnarray*} 
\nabla \cdot (D_0+D_1(\bx)) \nabla F^\eps &=&-\un_B(\bx) \sqrt{D_0+D_1(\bx)}\,T^\eps_V, \quad \bx \in \Rm^d, \\
F^\eps(\by)&=&\calO(|\by|^{1-d}) \qquad \textrm{ as }\quad |\by|\to \infty.
\end{eqnarray*}
The above equation is very similar to (\ref{Fprop}) at the end of
proof of proposition \ref{prop:jump} and a similar analysis yields
$$
\| \nabla F^\eps \|_{L^2(\Rm^d)}=\calO(\eps^{d+2}).
$$
We conclude the proof by calculating that
$$
\begin{array}{l}
\DST \sum_{|j|=1}^{d+1} \sum_{|i|=1}^{d+1} \frac{\eps^{|i|+|j|}}{i!j!} \partial^i N(\bx_0,\by) \partial^j V (\bx_0) \,\int_B D_1(\bx) \nabla \Big(\frac{\bx^j+\phi_j(\bx)}{\sqrt{D_0+D_1(\bx)}}  \Big)\cdot \nabla \bx^i d\bx,\\[4mm]
\DST \hspace{1cm}=\sum_{|j|=1}^{d+1} \sum_{|i|=1}^{d+1} \frac{\eps^{|i|+|j|}}{i!j!} \partial^i N(\bx_0,\by) \partial^j V (\bx_0) \frac{1}{\sqrt{D_0}} \int_B D_1(\bx) \nabla \Big(\bx^j+\psi_j+f_j \Big)\cdot \nabla \bx^i d\bx,\\[4mm]
\DST \hspace{1cm}=\frac{1}{\sqrt{D_0}}\sum_{|j|=1}^{d+1} \sum_{|i|=1}^{d+1} \frac{\eps^{|i|+|j|}}{i!j!} \partial^i N(\bx_0,\by) \partial^j V (\bx_0) M_{ij}+ \calO(\eps^{d+2}).
\end{array}
$$
\end{proofof}  

\subsection{Appendix}
This appendix states several lemmas that were needed in the preceding
analyses.
\begin{lemma} \label{lemappend1} 
  Let $\bF \in (L^2(\Rm^d))^d$ and $D_1 \in W^{1,\infty}(\Rm^d)$
  compactly supported in a bounded domain $B$, and $D_0$ a strictly
  positive constant. Assume moreover that $D_0+D_1(\bx)\geq C_0>0$
  \textit{a.e.} in $\Rm^d$. Then, the following problem (P):
\begin{eqnarray*}
\nabla \cdot \left(D_0+D_1(\bx) \right)\nabla \phi&=&\nabla \cdot \bF \qquad \textrm{ in  } \calD'(\Rm^d),\\
\phi(\bx)&=& \calO(|\bx|^{1-d})\quad \textrm{as} \quad |\bx|\to \infty, 
\end{eqnarray*}
admits unique solution in $H^1_\textrm{loc}(\Rm^d) \cap
\calC^\infty(\Rm^d \backslash \overline{B})$. Moreover, $\phi$
satisfies the estimates, for any bounded set $A \subset \Rm^d$,
\begin{equation} \label{estimappendphi}
\|\nabla \phi \|_{L^2(\Rm^d)} \leq C_0^{-1}\| \bF \|_{(L^2(B))^d},
 \quad \|\phi \|_{L^2(A)} \leq C \| \bF \|_{(L^2(B))^d} \left(1+\|D_1\|_{L^\infty(\Rm^d)} \right),
\end{equation}
and is the unique solution, $\textit{a.e.}$ on every bounded set of
$\Rm^d$, to the integral equation
\begin{eqnarray}\label{integappend}
D_0\, \phi(\by)&=& -\int_{B} D_1\left(\bx \right)\nabla \phi(\bx) \cdot   \nabla_\bx \Gamma(\bx-\by) d\bx+\int_{B} \bF(\bx) \cdot   \nabla_\bx \Gamma(\bx-\by) d\bx. 
\end{eqnarray}
\end{lemma}
\begin{proof} 
    We show that $(P)$ is equivalent to a problem posed on a bounded
  domain that can be solved with the Lax-Milgram lemma. To do
  so, let $B_R$ be the sphere of radius $R$ with $B \subset \subset B_R$
  and denote by $S_R$ its boundary. Consider the solution $\phi$ to
  (P) with the announced regularity. Since both $D_1$ and $\bF$ are
  supported in $B$, the function $\phi$ is harmonic in $\Rm^d
  \backslash \overline{B}$ and in particular in $\Rm^d \backslash
  \overline{B_R}$. Denoting by $\Lambda : H^{\frac12}(S_R) \to
  H^{-\frac12}(S_R) $ the exterior Dirichlet-Neumann operator on the
  sphere $S_R$, we then have the standard relation
  $$
  \frac{\partial \phi}{ \partial \bn}=\Lambda \left. \phi
  \right|_{S_R},
  $$
  where $\frac{\partial \phi}{ \partial \bn}$ is the outer normal
  derivative of $\phi$ on $S_R$ and $\left. \phi \right|_{S_R}$ its
  outer trace. Since $\phi$ is harmonic in $\Rm^d \backslash
  \overline{B}$ and is thus of class $\calC^\infty$ on this set,
  $\frac{\partial \phi}{ \partial \bn}$ and $\left. \phi
  \right|_{S_R}$ are continuous across $S_R$. Using this fact and
  integrating (P) against a test function $v \in
  \calC^\infty(\overline{B_R})$, we find
  $$\int_{B_R} (D_0+D_1)\nabla \phi \cdot \nabla v \, d\bx - D_0\langle
  \Lambda \left. \phi \right|_{S_R}, \left. v \right|_{S_R} \rangle
  =\int_B \bF \cdot \nabla v \,d\bx,$$
  where $\langle \cdot,
  \cdot\rangle$ denotes the $H^{\frac12}(S_R)-H^{-\frac12}(S_R)$ duality
  product. The restriction of $\phi$ to $B_R$ is therefore a solution to
  the following variational problem (P2): Find $u \in H^1(B_R)$
  such that
$$
a(u,v)=l(v), \qquad \forall v \in H^1(B_R),
$$
with obvious notation for the bilinear form $a$ and the linear
form $l$. Let us assume for the moment the existence of a unique
solution $u$ to (P2). That solution can be extended to a function
$u^*$ solution to $(P)$. Let indeed $u^*$ be defined as:
$$
\left\{
\begin{array}{l}
\DST u^*=u, \qquad \textrm{ in } B_R,\\
\DST u^*=U, \qquad \textrm{ in } \Rm^d \backslash \overline{B_R},
\end{array} \right.
$$
where $U$ is the solution to
\begin{eqnarray*}
\Delta U&=&0 \qquad \textrm{ in  } \calD'(\Rm^d\backslash \overline{B_R}),\\
\left. U \right|_{S_R}&=&\left. u \right|_{S_R},\qquad
U(\bx)\to 0\quad \textrm{as} \quad |\bx|\to \infty. 
\end{eqnarray*}
By construction, the trace of $u^*$ is continuous across $S_R$.
Since $U$ is harmonic in $\Rm^d \backslash \overline{B_R}$
and vanishes at infinity, it also verifies:
\begin{math}
  \frac{\partial U}{ \partial \bn}=\Lambda \left. U \right|_{S_R}=\Lambda \left. u \right|_{S_R}.
\end{math}
It then suffices to integrate the equation solved by $U$ against a
test function $v\in \calC^\infty_0(\Rm^d)$ and to consider (P2) to
find
$$\int_{\Rm^d} (D_0+D_1)\nabla u^* \cdot \nabla v \, d\bx =\int_B \bF
\cdot \nabla v \,d\bx, \qquad \forall v\in \calC^\infty_0(\Rm^d),$$
so
that $u^*$ solves (P). The above equation also implies that $u^*$ is
harmonic in $\Rm^d \backslash \overline{B}$ and is thus of class
$\calC^\infty$ on this set. It remains to verify the behavior at the
infinity, which stems from the fact that $\bF$ has compact support in
$B_R$.  Setting $v=1$ in (P2) yields $\langle \Lambda \left. u
\right|_{S_R}, 1 \rangle=0$. Getting back to $U$, since its trace and
its normal derivative are known and given by $\left. u \right|_{S_R}$
and $\Lambda \left. u \right|_{S_R}$, respectively, it admits the
following representation formula, for $\bx \in \Rm^d\backslash
\overline{B_R}$:
$$
U(\bx)=\int_{S_R} \left. u \right|_{S_R}(\by) \frac{\partial \Gamma(\bx-\by)}{\partial \bn_\by} d\sigma(\by)-\langle \Lambda \left. u\right|_{S_R}, \Gamma(\bx-\cdot)  \rangle,
$$
where $\Gamma$ is the fundamental solution of the Laplacian in
(\ref{eq:G}) and $\sigma$ is the surface measure on $S_R$. We conclude
by noticing that, as $|\bx| \to \infty$:
$$
\langle \Lambda \left. u\right|_{S_R}, \Gamma(\bx-\cdot) \rangle=\langle \Lambda \left. u\right|_{S_R}, \Gamma(\bx-\cdot)-\Gamma(\bx) \rangle=\calO(|\bx|^{1-d}).
$$

It remains to show the existence of a unique solution to (P2). This is
a consequence of the Lax-Milgram lemma: $a$ and $l$ are both
continuous in $H^1(B_R)$ and the coercivity follows from the
Poincar\'e-type inequality:
$$
\| u \|_{L^2(B_R)} \leq C \left( \|\nabla u \|_{L^2(B_R)} + \| u \|_{L^2(S_R)} \right), \qquad \forall u \in H^1(B_R),
$$
and the relation 
$$
C \| u \|^2_{L^2(S_R)}\leq - \langle \Lambda \left. u \right|_{S_R},  
\left. u \right|_{S_R} \rangle, \qquad \forall u \in H^{\frac12}(S_R).
$$

We now prove the first estimate in (\ref{estimappendphi}). Let $v \in
\calC^\infty_0(\Rm^d)$ such that
\begin{equation} \label{limitv}
\frac{1}{R^d}\|v\|_{L^1(S_R)} \to 0 \textrm{ as } R \to \infty.
\end{equation}
Integrating (P) against $v$ yields
$$\int_{B_R} (D_0+D_1)\nabla \phi \cdot \nabla v \, d\bx -
D_0\int_{S_R}\frac{\partial \phi}{ \partial \bn} v d\sigma=\int_B \bF
\cdot \nabla v \,d\bx.$$
Since $\nabla \phi(\bx)=\calO(|\bx|^{-d})$ as
$\bx$ tends to infinity, it belongs to $L^p(\Rm^d\backslash
\overline{B_\rho})$ for some $p>1$ and a ball of radius $\rho$ with $B
\subset \subset B_\rho$.  The above equality also holds by density for
all $v\in V_\rho$, the space of functions $v$ such that $v \in
H^1_\textrm{loc}(\Rm^d)$, $v$ verifies (\ref{limitv}) and $\nabla v
\in L^{p'}(\Rm^d\backslash \overline{B_\rho})$ for
$\frac{1}{p'}+\frac1p=1$. Since $\frac{\partial \phi}{ \partial
  \bn}=\calO(R^{-d}) $, sending $R$ to infinity implies, together with
(\ref{limitv}), that the boundary term goes to zero. On the other
hand, the function $\nabla \phi \cdot \nabla v$ is integrable on
$\Rm^d$ for $v\in V_\rho$, which allows us to use the Lebesgue
dominated convergence theorem and obtain as $R \to \infty$:
\begin{equation} \label{weakform} \int_{\Rm^d} (D_0+D_1)\nabla \phi \cdot \nabla v \, d\bx =\int_B \bF \cdot \nabla v \,d\bx,
\end{equation} 
for all $v\in V_\rho$. Since $\phi\in V_\rho$ for any $d\geq 2$, we
obtain the left estimate of (\ref{estimappendphi}).

Let us now consider the integral equation (\ref{integappend}) and show
that the solution to (P) verifies (\ref{integappend}). For $\psi \in
L^2(B_R)$, let $v(\bx)=\int_{B_R} \Gamma(\bx-\by) \psi(\by) d\by$ for
a given ball $B_R$. Since $\Gamma \in W^{1,1}_\textrm{loc}(\Rm^d)$, it
follows from the Young inequality that $v \in
H^1_\textrm{loc}(\Rm^d)$. Set $\bx \in \Rm^d \backslash
\overline{B_{R'}}$ with $B_R \subset \subset B_{R'}$. Then $\nabla
\Gamma(\cdot-\by) \in L^p(\Rm^d\backslash \overline{B_{R'}})$ for
$p>\frac{d}{d-1}$ and $\by \in B_R$. Such a function $v$ also
satisfies $(\ref{limitv})$ for $d\geq2$ since $\Gamma(\bx-\by)$ grows
at worst as $\log|\bx|$ for $(\bx,\by)\in \Rm^d \backslash
\overline{B_{R'}} \times B_R$. We can thus use $v$ as a test function
in $(\ref{weakform})$. In order to use the Fubini theorem, we notice
that the function $\nabla v(\bx) \cdot \nabla \Gamma(\bx-\by)
\psi(\by) \un_{B_R}(\by)$ belongs to $L^1(\Rm^d \times \Rm^d)$ thanks
to the Sobolev inequality \cite{RS-80-4} recalled in lemma
\ref{lem:sobolev} in the appendix since $\nabla v \in L^2(\Rm^d)$ and
$\psi \in L^2(B_R)$.  Indeed, since $R<\infty$, we bound the
$L^q(\Rm^d)$ norm of $\psi(\by) \un_{B_R}(\by)$ by the $L^2(B_R)$ norm
of $\psi$ for $q=\frac{2d}{d+2}\leq2$. Then choose $p=2$ and
$\lambda=d-1$ in lemma \ref{lem:sobolev}.

The same conclusion holds for $\bF(\bx) \cdot \nabla \Gamma(\bx-\by)
\psi(\by) \un_{B_R}(\by)$ so that we obtain from (\ref{weakform}):
$$ 
D_0\int_{B_R} \Big( \int_{\Rm^d} \nabla \phi(\bx) \cdot \nabla \Gamma(\bx-\by) \, d\bx \Big) \psi(\by) d\by
$$
\begin{equation} \hspace{1.5cm}\label{weakform2} 
=-\int_{B_R} \Big( \int_{\Rm^d} D_1(\bx)\nabla \phi(\bx) \cdot \nabla \Gamma(\bx-\by) -\bF(\bx) \cdot \nabla \Gamma(\bx-\by)\Big) \psi(\by) d\by.
\end{equation} 
It thus only remains to show that $\int_{\Rm^d} \nabla \phi(\bx) \cdot
\nabla \Gamma(\bx-\by) \, d\bx=\phi(\by)$ \textit{a.e.} on $B_R$ to
conclude. To this aim, consider a sequence $\phi^n$ of
$\calC^\infty_0(\Rm^d)$ functions such that $\nabla \phi^n \to \nabla
\phi$ in $L^2(\Rm^d)$ and $\phi^n \to \phi $ in $L^2(A)$ for any
bounded set $A$. Since $-\Delta_\bx \Gamma(\bx-\by)=\delta(\bx-\by)$
in the distribution sense, we have, for any $\by \in \Rm^d$:
$$
\lim_{\eps \to 0} \int_{|\bx-\by|>\eps} \Gamma(\bx-\by) \Delta \phi^n(\bx)d\bx=- \phi^n(\by).
$$
The Lebesgue dominated convergence theorem yields consequently:
$$
\lim_{\eps \to 0} \int_{B_R}\left(\int_{|\bx-\by|>\eps} \Gamma(\bx-\by) \Delta \phi^n(\bx)d\bx\right) \psi(\by) d\by=- \int_{B_R}\phi^n(\by) \psi(\by) d\by.
$$
An integration by parts then gives:
\begin{eqnarray*}
\int_{|\bx-\by|>\eps} \Gamma(\bx-\by) \Delta \phi^n(\bx)d\bx &=&\int_{|\bx-\by|=\eps}
\frac{\partial \phi^n(\bx)}{\partial \bn}  \Gamma(\bx-\by) d\sigma(\bx)\\
&&-\int_{|\bx-\by|>\eps} \nabla \Gamma(\bx-\by) \cdot \nabla \phi^n(\bx)d\bx.
\end{eqnarray*}
The boundary integral goes to zero with $\eps$. For the other term, we
remark that the function $\un_{|\bx-\by|>\eps}\un_{B_R} \nabla
\Gamma(\bx-\by) \cdot \nabla \phi^n(\bx) \psi(\by)$ converges
\textit{a.e.} in $\Rm^d \times \Rm^d$ to $\un_{B_R} \nabla
\Gamma(\bx-\by) \cdot \nabla \phi^n(\bx) \psi(\by)$ which belongs to
$L^1(\Rm^d \times \Rm^d)$ thanks to the Sobolev inequality. Applying
again the Lebesgue dominated convergence theorem yields
$$
\int_{B_R} \left( \int_{\Rm^d} \nabla \phi^n(\bx) \cdot \nabla \Gamma(\bx-\by) \, d\bx \right) \psi(\by) d\by=
\int_{B_R} \phi^n(\by) \psi(\by) d\by,
$$
and it suffices to pass to the limit in the sequence $\phi^n$ to
conclude. This proves that the solution to (P) satisfies
(\ref{integappend}). Conversely, considering a solution of
(\ref{integappend}) in $H^1_\textrm{loc}(\Rm^d) \cap
\calC^\infty(\Rm^d \backslash \overline{B})$, we verify using the same
techniques as above that this solution also satisfies (P), which we
know admits a unique solution. Therefore, the integral equation
(\ref{integappend}) also admits a unique solution. The second estimate
of $(\ref{estimappendphi})$ follows from $(\ref{integappend})$, the
Young inequality and the first estimate of $(\ref{estimappendphi})$.
\end{proof}
\bigskip

\begin{lemma} 
\label{lem:sobolev}
Sobolev inequality (see e.g. \cite{RS-80-4}). Let $f \in L^p(\Rm^d)$,
$g \in L^q(\Rm^d)$, $1<p,q<\infty$, $0<\lambda<d$ with the relation
$\frac{1}{p}+\frac{1}{q}+\frac{\lambda}{d}=2$. Then:
  $$
  \int_{\Rm^d}\int_{\Rm^d}\frac{|f(\bx)g(\by)|}{|\bx-\by|^{\lambda}}
  d\bx d\by \leq C \|f\|_{L^p(\Rm^d)} \|g\|_{L^q(\Rm^d)}.
$$
\end{lemma}
The following lemma, which is a standard variational formulation of
the Fredholm alternative, is used several times in the paper.
\begin{lemma} \label{lemappend}
  Let $H$ be a Hilbert space and let $a(\cdot,\cdot)$ be a bilinear
  form on a $H \times H$ such that 
  $a(\cdot,\cdot)=a_0(\cdot,\cdot)+a_1(\cdot,\cdot)$, where both $a_0$
  and $a_1$ are continuous in $H$ and $a_0$ is $H$-coercive. Assume
  moreover, that for two sequences $u_n$ and $v_n$ weakly converging
  in $H$ to $u$ and $v$, we have
  $$
  a_1(u_n,v_n) \to a_1(u,v).
  $$
  Then, if the following assertion is verified
$$
\left( a(u,v)=0,\quad \forall v \in H\right) \Longrightarrow u=0,
$$
for all $ f$ in $H'$, there exists a  unique $u \in H$ which satisfies
$$
a(u,v)=\langle f,v \rangle,\qquad \forall v \in H.
$$
Here, $\langle \cdot,\cdot \rangle$ denotes the $H'$-$H$ duality product.
Moreover, $u$ verifies the estimate, for some positive constant $C$:
$$
\|u\|_{H} \leq C \| f\|_{H'}.
$$  
\end{lemma}
\begin{proof}
  We sketch a proof for completeness.  Since $a_0$ is coercive, we
  know from the Lax-Milgram theory the existence of a bounded and
  boundedly invertible operator $S$ on $H$ such that
  $a_0(u,v)=(S^{-1}u,v)$, where $(\cdot, \cdot)$ is the inner product
  on $H$. By the Riesz representation theorem, we similarly know the
  existence of a bounded operator $A_1$ such that $a_1(u,v)=(A_1u,v)$.
  The hypotheses on $a_1$ imply that $A_1$ is compact on $H$. Indeed,
  choose $u_n \rightharpoonup u$ and define $v_n=A_1u_n-A_1u$. We
  verify that $v_n \rightharpoonup 0$ and that
  $\|A_1u_n-A_1u\|^2=(A_1u_n,v_n)-(A_1u,v_n)$ converges to $0$ by the
  above hypothesis on $a_1$ so that $A_1$ maps weakly converging
  sequences to strongly converging sequences and is thus compact.
  
  Now by the Riesz representation theorem, there exists $\tilde{f} \in
  H$ such that $\langle f,v \rangle=(\tilde{f},v)$, for all $v \in H$,
  so that $a(u,v)=\langle f,v \rangle$ is equivalent to
  $(S^{-1}+A_1)u=\tilde f$ and thus equivalent to $(I+SA_1)u=S \tilde
  f$, which admits a unique solution if and only if $-1$ is not an
  eigenvalue of the compact operator $SA_1$, which is equivalent to
  the fact that $a(u,v)=0$ for all $v\in H$ implies that $u=0$.
\end{proof}

\section*{Acknowledgments}
This work was funded in part by the National Science Foundation under
Grants DMS-0239097 and DMS-0554097.

{\footnotesize 

}    

\end{document}